\documentclass[reqno]{amsart}
\usepackage{amssymb,eucal}
\usepackage{graphicx}
\usepackage{epstopdf}

\usepackage{hyperref}
\hypersetup{backref=true,  
colorlinks=true,
citecolor=green,
linkcolor=green,
anchorcolor=green,
filecolor=blue,
menucolor=blue,
urlcolor=blue
}

\newcommand{\interp}{in\-ter\-po\-la\-tion}

\theoremstyle{plain}

\newtheorem{lemma}{Lemma}[section]
\newtheorem{theorem}[lemma]{Theorem}
\newtheorem{proposition}[lemma]{Proposition}
\newtheorem{corollary}[lemma]{Corollary}

\newtheorem*{stat}{\name}
\newcommand{\name}{testing}

\theoremstyle{definition}
\newtheorem{definition}[lemma]{Definition}
\newtheorem{example}[lemma]{Example}
\newtheorem{problem}{Problem}

\theoremstyle{remark}

\newtheorem{notation}[lemma]{Notation}
\newtheorem*{note}{Note}

\newcommand{\qedc}{{\qed}~{\rm Claim~{\theclaim}.}}
\newcommand{\qedsc}{{\qed}~{\rm Claim.}}

\numberwithin{equation}{section}
\numberwithin{figure}{section}

\newcommand{\pup}[1]{\textup{(}{#1}\textup{)}}
\newcommand{\jirr}{join-ir\-re\-duc\-i\-ble}
\newcommand{\mirr}{meet-ir\-re\-duc\-i\-ble}

\newcommand{\jsd}{join-sem\-i\-dis\-trib\-u\-tive}

\newcommand{\msd}{meet-sem\-i\-dis\-trib\-u\-tive}

\newcommand{\Msdy}{Meet-sem\-i\-dis\-trib\-u\-tiv\-i\-ty}

\newcommand{\contr}{a contradiction}

\newcommand{\ardef}{\underset{\scriptstyle\mathrm{def.}}{\Longleftrightarrow}}

\newcommand{\pI}[1]{\bigl({#1}\bigr)}

\newcommand{\set}[1]{\{#1\}}
\newcommand{\setm}[2]{\set{#1\mid#2}}

\newcommand{\seq}[1]{\langle{#1}\rangle}

\newcommand{\vecm}[2]{(#1\mid#2)}

\newcommand{\oce}[1]{\left]{#1}\right]_{\be}}
\newcommand{\coe}[1]{\left[{#1}\right[_{\be}}
\newcommand{\cce}[1]{\left[{#1}\right]_{\be}}

\newcommand{\so}[1]{\boldsymbol{\delta}_{#1}}

\newcommand{\orth}[1]{{#1}^{\perp}}
\newcommand{\oorth}[1]{{#1}^{\perp\perp}}
\DeclareMathOperator{\isol}{isol}
\DeclareMathOperator{\Con}{Con}
\DeclareMathOperator{\con}{con}
\DeclareMathOperator{\Bip}{Bip}
\DeclareMathOperator{\card}{card}

\DeclareMathOperator{\Clop}{Clop}
\DeclareMathOperator{\Reg}{Reg}
\DeclareMathOperator{\Regop}{Reg_{op}}
\DeclareMathOperator{\tcl}{cl}
\DeclareMathOperator{\tin}{int}

\DeclareMathOperator{\Ji}{Ji}
\DeclareMathOperator{\Mi}{Mi}
\newcommand{\pji}[3]{\seq{{#1},{#2};{#3}}}

\newcommand{\go}{\omega}
\newcommand{\gD}{\Delta}

\newcommand{\op}{\mathrm{op}}

\newcommand{\cF}{\mathcal{F}}

\newcommand{\Dr}{\mathbin{D}}

\newcommand{\dnw}{\mathbin{\downarrow}}

\newcommand{\tr}{\vartriangleleft}

\newcommand{\tre}{\vartriangleleft_{\be}}
\newcommand{\utre}{\trianglelefteq_{\be}}
\newcommand{\eqe}{\equiv_{\be}}

\DeclareMathOperator{\Pow}{Pow}
\newcommand{\cpl}{\mathsf{c}}

\newcommand{\jz}{$(\vee,0)$}

\newcommand{\res}{\mathbin{\restriction}}
\newcommand{\Res}[2]{#1\!\,\res_{\!\,#2}}
\newcommand{\es}{\varnothing}

\newlength{\cuplength}
\newcommand{\sA}{\mathsf{A}}
\newcommand{\sB}{\mathsf{B}}
\newcommand{\sL}{\mathsf{L}}
\newcommand{\sG}{\mathsf{G}}
\newcommand{\sK}{\mathsf{K}}
\newcommand{\sM}{\mathsf{M}}
\newcommand{\sN}{\mathsf{N}}
\newcommand{\sP}{\mathsf{P}}
\newcommand{\sR}{\mathsf{R}}
\newcommand{\sS}{\mathsf{S}}

\newcommand{\ba}{\boldsymbol{a}}
\newcommand{\bb}{\boldsymbol{b}}

\newcommand{\bc}{\boldsymbol{c}}

\newcommand{\be}{\boldsymbol{e}}

\newcommand{\bi}{\boldsymbol{i}}
\newcommand{\bp}{\boldsymbol{p}}
\newcommand{\bq}{\boldsymbol{q}}
\newcommand{\br}{\boldsymbol{r}}
\newcommand{\bu}{\boldsymbol{u}}

\newcommand{\bw}{\boldsymbol{w}}
\newcommand{\bx}{\boldsymbol{x}}
\newcommand{\by}{\boldsymbol{y}}
\newcommand{\bz}{\boldsymbol{z}}

\title[Extended permutohedron]{The extended permutohedron on a transitive binary relation}

\author[L. Santocanale]{Luigi Santocanale}
\address{Laboratoire d'Informatique Fondamentale de Marseille\\
Universit\'e de Provence\\
39 rue F. Joliot Curie\\
13453 Marseille Cedex 13\\
France}
\email{luigi.santocanale@lif.univ-mrs.fr}
\urladdr{http://www.lif.univ-mrs.fr/\~{}lsantoca/}

\author[F. Wehrung]{Friedrich Wehrung}
\address{LMNO, CNRS UMR 6139\\
D\'epartement de Math\'ematiques\\
Universit\'e de Caen\\
14032 Caen Cedex\\
France}
\email{wehrung@math.unicaen.fr}
\urladdr{http://www.math.unicaen.fr/\~{}wehrung}



\date{\today}

\begin{document}

\begin{abstract}
  For a given transitive binary relation~$\be$ on a set~$E$, the
  transitive closures of open (i.e., co-transitive in~$\be$) sets,
  called the \emph{regular closed} subsets, form an ortholattice $\Reg(\be)$, the \emph{extended permutohedron on~$\be$}.
  This construction, which contains the poset $\Clop(\be)$ of all clopen sets, is a common generalization of known notions such as the generalized permutohedron on a partially ordered set on the one hand, and the bipartition lattice on a set on the other hand. We obtain a precise description of the completely
  \jirr\ (resp., \mirr) elements of $\Reg(\be)$ and the arrow
  relations between them. In particular, we prove that
  \begin{itemize}
    \item[---] $\Reg(\be)$ is the Dedekind-MacNeille completion of the poset $\Clop(\be)$;
      \item[---] Every open subset of~$\be$ is a set-theoretic union of completely \jirr\ clopen subsets of~$\be$;
  \item[---] $\Clop(\be)$ is a lattice if{f} every regular closed subset of~$\be$ is clopen, if{f}~$\be$ contains no ``square'' configuration, if{f} $\Reg(\be)=\Clop(\be)$;
  \item[---] If~$\be$ is finite, then $\Reg(\be)$ is pseudocomplemented if{f} it is semidistributive, if{f} it is a bounded homomorphic image of a free lattice, if{f}~$\be$ is a disjoint sum of antisymmetric transitive relations and two-element full relations.
  \end{itemize}
We illustrate the strength of our results by proving that, for $n\geq3$, the congruence lattice of the lattice~$\Bip(n)$ of all bipartitions of an $n$-element set is obtained by adding a new top element to a Boolean lattice with $n\cdot 2^{n-1}$ atoms. We also determine the factors of the minimal subdirect decomposition of $\Bip(n)$.
\end{abstract}

\maketitle

%\tableofcontents

\section{Introduction}\label{S:Intro}
The lattice of all permutations $\sP(n)$ on an $n$-element chain, also known as the permutohedron,
even if widely known and studied in combinatorics, is a relatively young
object of study from a pure lattice theoretic perspective. Its
elements, the permutations of~$n$ elements, are endowed with the weak Bruhat order; this order turns out to be a lattice.

There are many possible generalization of this order, arising from the theory of Coxeter groups (Bj\"orner~\cite{Bjo83}), from graph and order theory (Pouzet \emph{et al.} \cite{PRRZ}, Hetyei and Krattenthaler~\cite{HetKra11}), from language
theory (Flath~\cite{Fla93}, Bennett and Birkhoff~\cite{BB94}).

In the present paper, we shall focus on one of the most noteworthy features---at least from the lattice-theoretical viewpoint---of one of the equivalent constructions of the permutohedron, namely that it can be realized as the lattice of all \emph{clopen} (i.e., both closed and open) subsets of a certain \emph{strict ordering} relation (viewed as a set of ordered pairs), endowed with the operation of transitive closure.

It turns out that most of the theory can be done for the transitive closure operator on the pairs of a given \emph{transitive} binary relation~$\be$. While, unlike the situation for ordinary permutohedra, the poset~$\Clop(\be)$ of all clopen subsets of~$\be$ may not be a lattice, it is contained in the larger lattice~$\Reg(\be)$ of all so-called \emph{regular closed} subsets of~$\be$, which we shall call the \emph{extended permutohedron on~$\be$} (cf. Section~\ref{S:RegCl}). As $\Reg(\be)$ is endowed with a natural orthocomplementation $\bx\mapsto\orth{\bx}$ (cf. Definition~\ref{D:OrthReg}), it becomes, in fact, an \emph{ortholattice}.
The natural question, whether $\Clop(\be)$ is a lattice, finds a natural answer in Theorem~\ref{T:SqFree}, where we prove that this is equivalent to the preordering associated with~$\be$ be \emph{square-free}, thus extending (with completely different proofs) known results for both the case of strict orderings (Pouzet \emph{et al.}~\cite{PRRZ}) and the case of full relations (Hetyei and Krattenthaler~\cite{HetKra11}).

However, while most earlier references 
deal with clopen subsets, our present paper focuses on the extended permutohedron~$\Reg(\be)$. One of our most noteworthy results is the characterization, obtained in Theorem~\ref{T:RegeSD}, of all finite transitive relations~$\be$ such that~$\Reg(\be)$ is \emph{semidistributive}. It turns out that this condition is equivalent to~$\Reg(\be)$ being \emph{pseudocomplemented}, also to~$\Reg(\be)$ being a \emph{bounded homomorphic image of a free lattice}, and can also be expressed in terms of forbidden sub-configurations of~$\be$. This result is achieved \emph{via} a precise description, obtained in Section~\ref{S:CJIRegP}, of all completely \jirr\ elements of~$\Reg(\be)$. This is a key technical point of the present paper. This description is further extended to a description of the join-dependency relation (cf. Section~\ref{S:JoinDep}), thus essentially completing the list of tools for proving one direction of Theorem~\ref{T:RegeSD}. The other directions are achieved \emph{via ad hoc} constructions, such as the one of Proposition~\ref{P:NonPseudo}.

Another noteworthy consequence of our description of completely \jirr\ elements of the lattice~$\Reg(\be)$ is the \emph{spatiality} of that lattice: every element is a join of completely \jirr\ elements\dots\ and even more can be said (cf. Lemma~\ref{L:OpenUjirr} and Theorem~\ref{T:JirrRege}). As a consequence, \emph{$\Reg(\be)$ is the Dedekind-MacNeille completion of~$\Clop(\be)$} (cf. Corollary~\ref{blue:C:MD}).

We proved in our earlier paper Santocanale and Wehrung~\cite{SaWe11} that the factors of the minimal subdirect decomposition of the permutohedron~$\sP(n)$ are exactly Reading's Cambrian lattices of type~A, denoted in~\cite{SaWe11} by $\sA_U(n)$. 
As a further application of our methods, we determine here the minimal subdirect decomposition of the lattice~$\Bip(n)$ of all bipartitions (i.e., those transitive binary relations with transitive complement) of an $n$-element set, thus solving the ``equation''
 \[
 \frac{\text{Tamari lattice}}
 {\text{permutohedron}}=
 \frac{x}{\text{bipartition lattice}}
 \]
and in fact, more generally,
 \begin{equation}\label{Eq:BipCambrian}
 \frac{\text{Cambrian lattice of type~A}}
 {\text{permutohedron}}=
 \frac{x}{\text{bipartition lattice}}\,.
 \end{equation}
The lattices~$x$ solving the ``equation'' \eqref{Eq:BipCambrian}, denoted here in the form $\sS(n,k)$ (cf. Section~\ref{S:SubdBip}), offer features quite different from those of the Cambrian lattices; in particular, they are not sublattices of the corresponding bipartition lattice~$\Bip(n)$, and their cardinality does not depend on~$n$ alone.

We also use our tools to determine the congruence lattice of every finite bipartition lattice (cf. Corollary~\ref{C:ConBipn}), which, for a base set with at least three elements, turns out to be Boolean with a top element added.

\section{Basic concepts and notation}
\label{S:NotaTerm}

We refer the reader to Gr\"atzer \cite{LTF} for basic facts, notation, and terminology about lattice theory.

We shall denote by~$0$ (resp., $1$) the least (resp., largest) element of a partially ordered set (from now on \emph{poset})~$(P,\leq)$, if they exist. A \emph{lower cover} of an element $p\in P$ is an element $x\in P$ such that $x<p$ and there is no~$y$ such that $x<y<p$. If~$p$ has a unique lower cover, then we shall denote this element by~$p_*$. \emph{Upper covers}, and the notation~$p^*$, are defined dually.

A nonzero element~$p$ in a lattice~$L$ is \emph{\jirr} if $p=x\vee y$
implies that $p\in\set{x,y}$, for all $x,y\in L$. We say that~$p$ is
\emph{completely \jirr} if it has a unique lower cover~$p_{*}$, and every $x < p$ satisfies $x \leq p_{*}$. \emph{Completely \mirr} elements are defined dually. We denote by~$\Ji L$ (resp., $\Mi L$) the set of all \jirr\ (resp., \mirr) elements of~$L$.

Every completely \jirr\ element is \jirr, and in a finite lattice, the two concepts are equivalent. A lattice~$L$ is \emph{spatial} if every
element of~$L$ is a (possibly infinite) join of completely \jirr\ elements. Equivalently, for all $a,b\in L$, $a\nleq b$ implies
that there exists a completely \jirr\ element~$p$ of~$L$ such that
$p\leq a$ and $p\nleq b$.

For a completely \jirr\ element~$p$ and a completely \mirr\ element~$u$ of~$L$, let $p\nearrow u$ hold if $p\leq u^*$ and $p\nleq u$. Symmetrically, let $u\searrow p$ hold if $p_*\leq u$ and $p\nleq u$. The \emph{join-dependency relation}~$\Dr$ is defined on completely \jirr\ elements by
 \[
 p\Dr q\ \ardef\ \pI{%
   \,p\neq q\text{ and }(\exists x)
   (p\leq q\vee x\text{ and }p\nleq q_*\vee x)
   \,}\,.
 \]
It is well-known (cf. Freese, Je\v{z}ek, and Nation \cite[Lemma~11.10]{FJN}) that the join-dependency relation~$\Dr$ on a finite lattice~$L$ can be conveniently expressed in terms of the arrow relations~$\nearrow$ and~$\searrow$ between~$\Ji L$ and~$\Mi L$.

\begin{lemma}\label{L:Arr2D}
Let $p$, $q$ be distinct \jirr\ elements in a finite lattice~$L$. Then $p\Dr_Lq$ if{f} there exists $u\in\Mi L$ such that $p\nearrow u\searrow q$.
\end{lemma}

We shall denote by~$\Dr^n$ (resp., $\Dr^*$) the $n$th relational power (resp., the reflexive and transitive closure) of the~$\Dr$ relation, so, for example, $p\Dr^2 q$ if{f} there exists $r\in\Ji L$ such that $p\Dr r$ and $r\Dr q$.

It is well-known that the congruence lattice~$\Con L$ of a finite
lattice~$L$ can be conveniently described \emph{via} the~$\Dr$
relation on~$L$, as follows (cf. Freese, Je\v{z}ek, and Nation
\cite[Section~II.3]{FJN}). Denote by~$\con(p)$ the least congruence
of~$L$ containing $(p_*,p)$ as an element, for each $p\in\Ji L$. Then
$\con(p)\subseteq\con(q)$ if{f} $p\Dr^*q$, for all $p,q\in\Ji
L$. Furthermore, $\Con L$ is a distributive lattice and $\Ji(\Con
L)=\setm{\con(p)}{p\in\Ji L}$. A subset $S\subseteq\Ji L$ is a
\emph{$\Dr$-upper subset} if $p\in S$ and $p\Dr q$ implies that $q\in
S$, for all $p,q\in\Ji L$. Set $S\dnw x=\setm{s\in S}{s\leq x}$, for
each $x\in L$.

\begin{lemma}\label{L:D2ConL}
The binary relation $\theta_S=\setm{(x,y)\in L\times L}{S\dnw x=S\dnw y}$ is a congruence of~$L$, for every finite lattice~$L$ and every $\Dr$-upper subset~$S$ of $\Ji L$, and the assignment $x\mapsto x/{\theta_S}$ defines an isomorphism from the \jz-subsemilattice~$S^\vee$ of~$L$ generated by~$S$ onto the quotient lattice $L/{\theta_S}$. Furthermore, the assignment $S\mapsto\theta_S$ defines a dual isomorphism from the lattice of all $\Dr$-upper subsets of~$\Ji L$ onto~$\Con L$. The inverse of that isomorphism is given by
 \[
 \theta\mapsto\setm{p\in\Ji L}
 {(p,p_*)\notin\theta}\,,\quad
 \text{for each }\theta\in\Con L\,.
 \]
\end{lemma}

For each $p\in\Ji L$, denote by~$\Psi(p)$ the largest congruence $\theta$ of~$L$ such that $p\not\equiv p_*\pmod{\theta}$. Then $\Psi(p)=\theta_{S_p}$ where we set $S_p=\setm{q\in\Ji L}{p\Dr^*q}$. Equivalently, $\Psi(p)$ is generated by all pairs $(q,q_{*})$ such that $p \Dr^* q$ does not hold. Say that $p \in \Ji L$ is \emph{$\Dr^*$-minimal} if $p \Dr^* q$ implies $q \Dr^* p$, for each $q\in\Ji L$. The set~$\gD(L)$ of all $\Dr^*$-minimal \jirr\ elements of~$L$ defines, \emph{via}~$\Psi$, a subdirect product decomposition of~$L$,
 \begin{equation}\label{Eq:MinSubProd}
 L\hookrightarrow
 \prod\vecm{L/{\Psi(p)}}{p\in\gD(L)}\,,\quad
 x\mapsto\vecm{x/{\Psi(p)}}{p\in\gD(L)}\,,
 \end{equation}
that we shall call the \emph{minimal subdirect product decomposition of~$L$}.

A lattice~$L$ is \emph{\jsd} if $x\vee z=y\vee z$ implies that $x\vee
z=(x\wedge y)\vee z$, for all $x,y,z\in L$. \emph{\Msdy} is defined
dually. A lattice is \emph{semidistributive} if it is both join- and
\msd.

A lattice~$L$ is a \emph{bounded homomorphic image of a free lattice}
if there are a free lattice~$F$ and a surjective lattice homomorphism
$f\colon F\twoheadrightarrow L$ such that~$f^{-1}\set{x}$ has both a
least and a largest element, for each $x\in L$. These lattices,
introduced by McKenzie~\cite{McKe72}, form a quite important class within the theory of lattice
  varieties, and are often called ``bounded lattices'' (not to be
confused with lattices with both a least and a largest element). A
finite lattice is bounded if{f} the join-dependency relations on~$L$
and on its dual lattice are both cycle-free (cf. Freese, Je\v{z}ek,
and Nation \cite[Corollary~2.39]{FJN}). Every bounded lattice is
semidistributive (cf. Freese, Je\v{z}ek, and Nation
\cite[Theorem~2.20]{FJN}), but the converse fails, even for finite
lattices (cf. Freese, Je\v{z}ek, and Nation \cite[Figure~5.5]{FJN}).

An \emph{orthocomplementation} on a poset~$P$ with least and largest element is a map $x\mapsto\orth{x}$ of~$P$ to itself such that
\begin{itemize}
\item[(O1)] $x\leq y$ implies that $\orth{y}\leq\orth{x}$,

\item[(O2)] $\oorth{x}=x$,

\item[(O3)] $x\wedge\orth{x}=0$ (in view of~(O1) and~(O2), this is equivalent to $x\vee\orth{x}=1$),
\end{itemize}
for all $x,y\in P$. Elements $x,y\in P$ are \emph{orthogonal} if $x\leq\orth{y}$, equivalently $y\leq\orth{x}$.

An \emph{orthocomplemented poset} is a poset with an
orthocomplementation. Of course, any orthocomplementation of~$P$ is a
dual automorphism of~$(P,\leq)$. In particular, if~$P$ is a lattice,
then \emph{de Morgan's rules}
 \[
 \orth{(x\vee y)}=\orth{x}\wedge\orth{y}\,,\quad
 \orth{(x\wedge y)}=\orth{x}\vee\orth{y}
 \]
hold for all $x,y\in P$. An \emph{ortholattice} is a lattice endowed with an orthocomplementation.

  A lattice~$L$ with a least element~$0$ is \emph{pseudocomplemented} if $\setm{y \in L}{x \wedge y = 0}$ has a greatest element, for each $x \in P$. 

We shall denote by~$\Pow X$ the powerset of a set~$X$, and we shall set $\Pow^*X=(\Pow X)\setminus\set{\es,X}$. We shall also set $[n]=\set{1,2,\dots,n}$, for every positive integer~$n$.

\section{Regular closed subsets of a transitive relation}\label{S:RegCl}

Unless specified otherwise, by ``relation'' (on a set) we shall always mean a binary relation. For a relation~$\be$ on a set~$E$, we will often write
 \begin{align*}
 x\tre y&\ \ardef\ (x,y)\in\be\,,\\
 x\utre y&\ \ardef\ 
 (\text{either }x\tre y\text{ or }x=y)\,,\\
 x\eqe y&\ \ardef\ (x\utre y
 \text{ and }y\utre x)\,,
 \end{align*}
for all $x,y\in E$. We also set
 \begin{align*}
 \cce{a,b}&=\setm{x}{a\utre x
 \text{ and }x\utre b}\,,\\
 \coe{a,b}&=\setm{x}{a\utre x
 \text{ and }x\tre b}\,,\\
 \oce{a,b}&=\setm{x}{a\tre x
 \text{ and }x\utre b}\,,\\
 \cce{a}&=\cce{a,a}\,,
 \end{align*}
for all $a,b\in E$. As $a\tre a$ may occur, $a$ may belong to $\oce{a,b}$.

Denote by~$\tcl(\ba)$ the transitive closure of any relation~$\ba$. We say that~$\ba$ is \emph{closed} if it is transitive. We say that~$\ba$ is \emph{bipartite} if there are no~$x$, $y$, $z$ such that $(x,y)\in\ba$ and $(y,z)\in\ba$. It is trivial that every bipartite relation is closed.

Let~$\be$ be a transitive relation on a set~$E$. A subset~$\ba\subseteq\be$ is \emph{open} (\emph{relatively to~$\be$}) if $\be\setminus\ba$ is closed; equivalently,
 \[
 \pI{x\tre y\tre z
 \text{ and }(x,z)\in\ba}\Rightarrow
 \pI{\text{either }(x,y)\in\ba
 \text{ or }(y,z)\in\ba}\,,\ \text{for all }
 x,y,z\in E\,.
 \]
 The largest open subset of $\ba\subseteq\be$, called the
 \emph{interior} of~$\ba$ and denoted by $\tin(\ba)$, is
 exactly the set of all pairs $(x,y)\in\be$ such that for every
 subdivision $x=z_0\tre z_1\tre \cdots\tre z_n=y$, with $n>0$, there
 exists $i<n$ such that $(z_i,z_{i+1})\in\ba$. We shall repeatedly use
 the easy observation that both operators $\tcl\circ\tin$ and
 $\tin\circ\tcl$ are idempotent.

A subset $\ba\subseteq\be$ is \emph{clopen} if $\ba=\tcl(\ba)=\tin(\ba)$.  We denote by $\Clop(\be)$ the poset of all clopen subsets of~$\be$.  A subset $\ba\subseteq\be$ is \emph{regular closed} (resp., \emph{regular open}) if $\ba=\tcl\tin(\ba)$ (resp., $\ba=\tin\tcl(\ba)$).  We denote by $\Reg(\be)$ (resp., $\Regop(\be)$) the poset of all regular closed (resp., regular open) subsets.

As a set~$\bx$ is open if{f} its complement $\bx^\cpl=\be\setminus\bx$ is closed (by definition), similarly a set~$\bx$ is closed (regular closed, regular open, clopen, respectively) if{f}~$\bx^\cpl$ is open (regular open, regular closed, clopen, respectively).

The proof of the following lemma is a straightforward exercise.

\begin{lemma}\label{L:BasicReg}
\hfill
\begin{enumerate}
\item A subset~$\bx$ of~$\be$ is regular closed if{f} $\bx=\tcl(\bu)$ for some open set~$\bu$.

\item The poset~$\Reg(\be)$ is a complete lattice, with meet and join given by
 \begin{align*}
 \bigvee\vecm{\ba_i}{i\in I}&=\tcl
 \pI{\bigcup\vecm{\ba_i}{i\in I}}\,,\\
 \bigwedge\vecm{\ba_i}{i\in I}&=\tcl\tin
 \pI{\bigcap\vecm{\ba_i}{i\in I}}\,, 
 \end{align*}
for any family $\vecm{\ba_i}{i\in I}$ of regular closed sets.
\end{enumerate}
\end{lemma}

The complement of a regular closed set may not be closed. Nevertheless, we shall now see that there is an obvious ``complementation-like'' map from the regular closed sets to the regular closed sets.

\begin{definition}\label{D:OrthReg}
We define the \emph{orthogonal} of~$\bx$ as $\orth{\bx}=\tcl(\be\setminus\bx)$, for any $\bx\subseteq\be$.
\end{definition}

A straightforward use of Lemma~\ref{L:BasicReg}(i) yields the following lemma.

\begin{lemma}\label{L:BasicOrth}
\hfill
\begin{enumerate}
\item $\orth{\bx}$ is regular closed, for any $\bx\subseteq\be$.

\item The assignment $\orth{}\colon\bx\mapsto\orth{\bx}$ defines an orthocomplementation of~$\Reg(\be)$.
\end{enumerate}
\end{lemma}

In particular, $\Reg(\be)$ is self-dual. As $\bx\mapsto\bx^\cpl$ defines a dual isomorphism from $\Reg(\be)$ to $\Regop(\be)$, we obtain the following.

\begin{corollary}\label{C:RegP2RegopP}
The lattices~$\Reg(\be)$ and~$\Regop(\be)$ are pairwise isomorphic, and also self-dual, for any transitive relation~$\be$.
\end{corollary}

We shall call $\Clop(\be)$ the \emph{permutohedron on~$\be$} and~$\Reg(\be)$ the \emph{extended permutohedron on~$\be$}. For example, if~$\be$ is the strict ordering associated to a poset~$(E,\leq)$, then $\Clop(\be)$ is the poset denoted by~$\mathbf{N}(E)$ in Pouzet \emph{et al.} \cite{PRRZ}. On the other hand, if~$\be=[n]\times[n]$ for a positive integer~$n$, then $\Clop(\be)$ is the poset of all \emph{bipartitions} of~$[n]$ introduced in Foata and Zeilberger~\cite{FoZe96} and Han~\cite{Han96}, see also Hetyei and Krattenthaler~\cite{HetKra11} where this poset is denoted by~$\Bip(n)$.

While the lattice~$\Reg(\be)$ is always orthocomplemented (cf. Lemma~\ref{L:BasicOrth}), the following result shows that~$\Reg(\be)$ is not always pseudocomplemented.

\begin{proposition}\label{P:NonPseudo}
  Let~$\be$ be a transitive relation on a set~$E$ with pairwise
  distinct elements $a_0,a_1,b\in E$ such that $a_0\eqe a_1$ and either $b\tre a_0$ or $a_0\tre b$. Then there are clopen subsets~$\ba_0$, $\ba_1$, $\bc$ of~$\be$ such that $\ba_0\wedge\bc=\ba_1\wedge\bc=\es$ while $\es\neq\bc\subseteq\ba_0\vee\ba_1$.  In particular, the lattice $\Reg(\be)$ is neither \msd, nor pseudocomplemented.
\end{proposition}

\begin{proof}
We show the proof in the case where $a_0\tre b$. By applying the result to $\be^{\op}=\setm{(x,y)}{(y,x)\in\be}$, the result for the case $b\tre a_0$ will follow. 
We set $I=\cce{a_0,b}=\cce{a_1,b}$ and
 \begin{align*}
 \ba_i&=
 \set{a_i}\times(I\setminus\set{a_i})
 &&(\text{for each }i\in\set{0,1})\,,\\
 \bc&=\set{a_0,a_1}\times
 (I\setminus\set{a_0,a_1})\,. 
 \end{align*}
It is straightforward to verify that~$\ba_0$, $\ba_1$, $\bc$ are all clopen subsets of~$\be$. Furthermore, $(a_0,b)\in\bc$ thus $\bc\neq\es$, and $\bc\subseteq\ba_0\cup\ba_1\subseteq\ba_0\vee\ba_1$.

Now let $(a_0,x)$ be an element of $\ba_0\cap\bc=\set{a_0}\times(I\setminus\set{a_0,a_1})$. Observing that $a_0\tre a_1\tre x$ while $(a_0,a_1)\notin\ba_0\cap\bc$ and $(a_1,x)\notin\ba_0\cap\bc$, we obtain that $(a_0,x)\notin\tin\pI{\ba_0\cap\bc}$; whence $\ba_0\wedge\bc=\es$. Likewise, $\ba_1\wedge\bc=\es$.
\end{proof}

\section{Lattices of clopen subsets of square-free transitive relations}\label{S:SqFree}

\begin{definition}\label{D:SqFree}
A transitive relation~$\be$ is \emph{square-free} if for all $(a,b)\in\be$, any two elements of $\cce{a,b}$ are comparable with respect to~$\utre$. That is,
 \begin{multline*}
 (\forall a,b,x,y)\Bigl(
 \pI{a\utre x\text{ and }a\utre y
 \text{ and }x\utre b\text{ and }y\utre b}\\
 \Longrightarrow
 (\text{either }x\utre y\text{ or }
 y\utre x)\Bigr)\,.
 \end{multline*}
\end{definition}

For the particular case of the natural strict ordering $1<2<\cdots<n$,
the following result originates in Guilbaud and Rosenstiehl
\cite[Section~VI.A]{GuRo63}. The case of the full relation $[n]\times[n]$ is covered
by the proof of Hetyei and
Krattenthaler~\cite[Proposition~4.2]{HetKra11}.

\begin{lemma}\label{L:Int(closed)be}
Let~$\be$ be a square-free transitive relation. Then the set $\tin(\ba)$ is closed, for each closed $\ba\subseteq\be$. Dually, the set $\tcl(\ba)$ is open, for each open $\ba\subseteq\be$.
\end{lemma}

\begin{proof}
  It suffices to prove the
  first statement. Let $x\tre y\tre z$ with $(x,y)\in\tin(\ba)$ and
  $(y,z)\in\tin(\ba)$, we must prove that $(x,z)\in\tin(\ba)$.
  Consider a subdivision $x=s_0\tre s_1\tre\cdots\tre s_n=z$ and
  suppose that
 \begin{equation}\label{Eq:sinowita}
 (s_i,s_{i+1})\notin\ba\text{ for each }i<n
 \end{equation}
 (we say that the subdivision \emph{fails
   witnessing~$(x,z)\in\tin(\ba)$}). Denote by~$l$ the largest integer
 such that $l<n$ and $s_l\utre y$. If $s_l=y$, then the subdivision
 $x=s_0\tre s_1\tre\cdots\tre s_l=y$ fails
 witnessing~$(x,y)\in\tin(\ba)$, \contr; so $s_l\neq y$ and $s_l\tre y$. {}From $x=s_0\tre s_1\tre\cdots\tre s_l\tre y$,
 $(x,y)\in\tin(\ba)$, and~\eqref{Eq:sinowita} it follows that
 \begin{equation}\label{Eq:slyina}
 (s_l,y)\in\ba\,.
 \end{equation}
 As~$\be$ is square-free, either $s_{l+1}\utre y$ or $y\tre s_{l+1}$. In the first case, it follows from the definition of~$l$ that
 $l=n-1$, thus, using~\eqref{Eq:slyina} together with $(y,z)\in\ba$, we get $s_{n-1}=s_l\tr_{\ba}y\tr_{\ba}z = s_{n}$, whence $(s_{n-1},s_{n}) \in \ba$, which contradicts~\eqref{Eq:sinowita}.
  
Hence $y\tre s_{l+1}$. {}From $y\tre s_{l+1}\tre s_{l+2}\tre\cdots\tre s_n=z$, $(y,z)\in\tin(\ba)$, and~\eqref{Eq:sinowita} it follows that $(y,s_{l+1})\in\ba$, thus, by~\eqref{Eq:slyina}, $(s_l,s_{l+1})\in\ba$, in contradiction
 with~\eqref{Eq:sinowita}.
\end{proof}

In the particular case of \emph{strict orderings} (i.e., irreflexive transitive relations), most of the following result is contained (with a completely different argument) in Pouzet \emph{et al.}~\cite[Lemma~12]{PRRZ}.

\begin{theorem}\label{T:SqFree}
The following are equivalent, for any transitive relation~$\be$:
\begin{enumerate}
\item $\be$ is square-free;

\item $\Clop(\be)=\Reg(\be)$;

\item $\Clop(\be)$ is a lattice;
  
\item $\Clop(\be)$ has the \emph{\interp\ property}, that is, for all $\bx_0,\bx_1,\by_0,\by_1\in\Clop(\be)$ such that $\bx_i\subseteq\by_j$ for all $i,j<2$, there exists $\bz\in\Clop(\be)$ such that $\bx_i\subseteq\bz$ and $\bz\subseteq\by_i$ for all $i<2$.
\end{enumerate}
\end{theorem}

\begin{proof}
(i)$\Rightarrow$(ii) follows immediately from Lemma~\ref{L:Int(closed)be}.

(ii)$\Rightarrow$(iii) and (iii)$\Rightarrow$(iv) are both trivial.
  
(iv)$\Rightarrow$(i). We prove that if~$\be$ is not square-free, then~$\Clop(\be)$ does not satisfy the \interp\ property. By assumption, there are $(a,b)\in\be$ and $u,v\in\cce{a,b}$ such that $u\not\utre v$ and $v\not\utre u$. It is easy to verify that the subsets
 \begin{align*}
 \bx_0&=\set{a}\times\oce{a,u}\,,&
 \bx_1&=\coe{u,b}\times\set{b}\,,\\
 \by_0&=(\set{a}\times\oce{a,b})
 \cup\bx_1\,,&
 \by_1&=(\coe{a,b}\times\set{b})
 \cup\bx_0
 \end{align*}
are all clopen, and that $\bx_i\subseteq\by_j$ for all $i,j<2$. Suppose that there exists $\bz\in\Clop(\be)$ such that $\bx_i \subseteq\bz \subseteq\by_i$ for each $i<2$. {}From $(a,u)\in\bx_0\subseteq\bz$ and $(u,b)\in\bx_1\subseteq\bz$ and the transitivity of~$\bz$ it follows that $(a,b)\in\bz$, thus, as $a\tre v\tre b$ and~$\bz$ is open, either $(a,v)\in\bz$ or $(v,b)\in\bz$. In the first case, $(a,v)\in\by_1$, thus $v\utre u$, a contradiction. In the second case, $(v,b)\in\by_0$, thus $u\utre v$, a contradiction.
\end{proof}

By applying Theorem~\ref{T:SqFree} to the full relation $[n]\times[n]$ (which is trivially square-free), we obtain
the following result, first proved in Hetyei and
Krattenthaler~\cite[Theorem~4.1]{HetKra11}.

\begin{corollary}[Hetyei and Krattenthaler]
\label{C:HetKra}
The poset~$\Bip(n)$ of all bipartitions of~$[n]$ is a lattice.
\end{corollary}

\begin{example}\label{Ex:PermPoset2}
We set $\so{E}=\setm{(x,y)\in E\times E}{x<y}$, for any poset~$E$, and we set $\sP(E)=\Clop(\so{E})$ and $\sR(E)=\Reg(\so{E})$. By Theorem~\ref{T:SqFree} (see also Pouzet \emph{et al.} \cite[Lemma~12]{PRRZ}), $\sP(E)$ is a lattice if{f}~$E$ contains no copy of the four-element Boolean lattice~$\sB_2=\set{0,a,b,1}$ (represented on the left hand side diagram of Figure~\ref{Fig:PB2})---that is, by using the above terminology, $\so{E}$ is \emph{square-free}.

\begin{figure}[htb]
\includegraphics{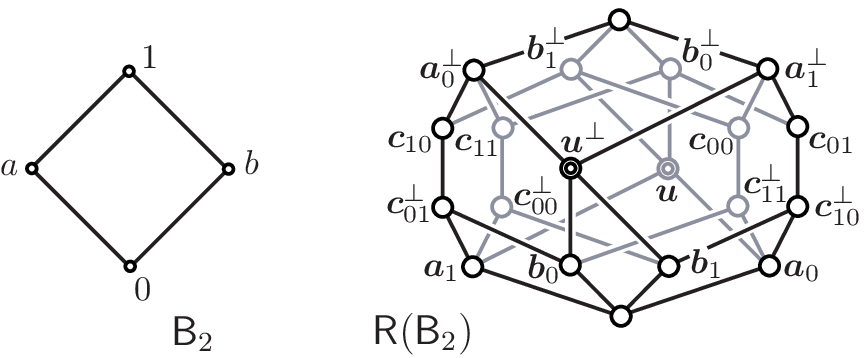}
\caption{The lattice $\sR(\sB_2)$}\label{Fig:PB2}
\end{figure}

The lattice~$\sR(\sB_2)$ has~$20$ elements, while its
subset~$\sP(\sB_2)$ has~$18$ elements. The lattice~$\sR(\sB_2)$ is
represented on the right hand side of Figure~\ref{Fig:PB2}. Its \jirr\ 
elements, all clopen (see a general explanation in
Theorem~\ref{T:JirrRege}), are
\begin{gather*}
\ba_0=\set{(0,a)}\,,\quad\ba_1=\set{(a,1)}\,,
\quad\bb_0=\set{(0,b)}\,,\quad\bb_1=\set{(b,1)}\,,\\
\bc_{00}=\set{(0,a),(0,b),(0,1)}\,,\quad
\bc_{01}=\set{(0,a),(b,1),(0,1)}\,,\\
\bc_{10}=\set{(a,1),(0,b),(0,1)}\,,\quad
\bc_{11}=\set{(a,1),(b,1),(0,1)}\,.
\end{gather*}
The two elements of $\sR(\sB_2)\setminus\sP(\sB_2)$ are
$\bu=\set{(0,a),(a,1),(0,1)}$ together with its orthogonal,
$\orth{\bu}=\set{(0,b),(b,1),(0,1)}$. Those elements are marked by
doubled circles on the right hand side diagram of
Figure~\ref{Fig:PB2}.
\end{example}

\begin{example}\label{Ex:M3inBip3}
It follows from Proposition~\ref{P:NonPseudo} that the lattice~$\Bip(3)$ of all bipartitions of~$[3]$ is not pseudocomplemented.
We can say more: $\Bip(3)$ contains a copy of the five element lattice~$\sM_3$ of length two, namely $\set{\es,\ba,\bb,\bc,\be}$, where
 \begin{align*}
 \ba&=\set{(1,2),(3,1),(3,2)}\,,\\
 \bb&=\set{(2,1),(2,3),(3,1)}\,,\\
 \bc&=\set{(1,2),(1,3),(2,3)}\,,\\
 \be&=[3]\times[3]\,.
 \end{align*}
It is also observed in Hetyei and Krattenthaler
\cite[Example~7.7]{HetKra11} that~$\Bip(3)$ contains a copy of the
five element nonmodular lattice~$\sN_5$; hence it is not modular.
\end{example}

\section{Completely \jirr\ clopen sets}\label{S:CJIRegP}

Throughout this section we shall fix a transitive relation~$\be$ on a set~$E$.

\begin{definition}\label{D:abU}
We denote by~$\cF(\be)$ the set of all triples $(a,b,U)$, where $(a,b)\in\be$, $U\subseteq\cce{a,b}$, and $a\neq b$ implies that $a\notin U$ and $b\in U$. We set $U^\cpl=\cce{a,b}\setminus U$, and
 \[
 \pji{a}{b}{U}=\begin{cases}
 \setm{(x,y)}{a\utre x\tre y\utre b\,,
 \ x\notin U\,,\text{ and }y\in U}\,,&
 \text{if }a\neq b\,,\\
 (\set{a}\cup U^\cpl)\times
 (\set{a}\cup U)\,,&
 \text{if }a=b\,,
 \end{cases}
 \]
for each $(a,b,U)\in\cF(\be)$. Hence 
$\pji{a}{b}{U}=\be\cap\pI{(\set{a}\cup U^\cpl)\times(\set{b}\cup U)}$ in each case.

Observe that $\pji{a}{b}{U}$ is bipartite if{f} $a\neq b$. If $a=b$,
we shall say that $\pji{a}{b}{U}$ is a \emph{clepsydra}.
\footnote{After ``clessidra'', which is the Italian for ``hourglass'',
  the latter describing the pattern of the associated transitive
  relation: the elements of~$U^\cpl$ below; $a$ in the middle; the
  elements of~$U$ above.}
\end{definition}

The proof of the following lemma is a straightforward exercise.

\begin{lemma}\label{L:TwoJirrEq}
Let $(a,b,U),(c,d,V)\in\cF(\be)$. Then $\pji{a}{b}{U}=\pji{c}{d}{V}$ if{f} one of the following statements occurs:
\begin{enumerate}
\item $a\neq b$, $c\neq d$, $a\eqe c$, $b\eqe d$, and $U=V$;
\item $a=b=c=d$ and $U\setminus\set{a}=V\setminus\set{a}$.
\end{enumerate}
\end{lemma}

\begin{lemma}\label{L:abUClopJirr}
The set $\bp=\pji{a}{b}{U}$ is clopen and $(a,b)\in\bp$, for each $(a,b,U)\in\cF(\be)$. Furthermore, the set~$\bp_*$ defined by
 \begin{equation}\label{Eq:Defp*}
 \bp_*=\begin{cases}
 \bp\setminus(\cce{a}\times\cce{b})\,,
 &\text{if }a\neq b\,,\\
 \bp\setminus\set{(a,a)}\,,
 &\text{if }a=b
 \end{cases}
 \end{equation}
is clopen, and every proper open subset of~$\bp$ is contained in~$\bp_*$.
\end{lemma}

\begin{note}
The notation~$\bp_*$ will be validated shortly, in Corollary~\ref{C:abUClopJirr}, by proving that~$\bp_*$ is, indeed, the unique lower cover of~$\bp$ in the lattice~$\Reg(\be)$.
\end{note}

\begin{proof}
In both cases it is trivial that $(a,b)\in\bp$.

Now consider the case where $a\neq b$. In that case, $\bp$ is bipartite, thus closed. Let $x\tre y\tre z$ with $(x,z)\in\bp$. If $y\in U$, then $(x,y)\in\bp$, and if $y\notin U$, then $(y,z)\in\bp$. Hence~$\bp$ is clopen.

As $\bp_*\subseteq\bp$ and~$\bp$ is bipartite, $\bp_*$ is bipartite as well, thus~$\bp_*$ is closed. Let $x\tre y\tre z$ with $(x,z)\in\bp_*$, and suppose by way of contradiction that $(x,y)\notin\bp_*$ and $(y,z)\notin\bp_*$. As~$\bp$ is open, either $(x,y)\in\bp$ or $(y,z)\in\bp$, hence either $(x,y)$ or $(y,z)$ belongs to $\bp\cap(\cce{a}\times\cce{b})$. In the first case, $x\eqe a$ and $x\notin U$. Furthermore, $b\eqe y\tre z$, but $z\utre b$ (because $(x,z)\in\bp_*\subseteq\bp$), so $z\eqe b$, and so we get $(x,z)\in\bp\cap(\cce{a}\times\cce{b})=\bp\setminus\bp_*$, \contr. The second case is dealt with similarly. Therefore, $\bp_*$ is open.

Let $\bu\subseteq\bp$ be open and suppose that~$\bu$ is not contained
in~$\bp_*$. This means that there exists $(a',b')\in\bu$ such that
$a\eqe a'$ and $b\eqe b'$. We must prove that $\bp\subseteq\bu$. Let
$(x,y)\in\bp$, so that $x\notin U$ and $y\in U$. {}From $x\notin U$ it follows that $(a',x)\notin\bp$, thus $(a',x)\notin\bu$; whence, from $a'\utre x\tre b'$, $(a',b')\in\bu$, and the openness of~$\bu$, we get $(x,b')\in\bu$.
Now $y\in U$, thus $(y,b')\notin\bp$, and thus $(y,b')\notin\bu$, hence, as $x\tre y\utre b'$, $(x,b')\in\bu$, and~$\bu$ is open, we get
$(x,y)\in\bu$, as required.

{}From now on suppose that $a=b$. It is trivial that~$\bp$ is closed (although it is no longer bipartite). The proof that~$\bp$ is open is similar to the one for the case where $a\neq b$.

Let $x\tre y\tre z$ with $(x,y)\in\bp_*$ and $(y,z)\in\bp_*$. {}From $\bp_*\subseteq\bp$ it follows that $y\in\set{a}\cup U^\cpl$ and $y\in\set{a}\cup U$, thus $y=a$, and thus (as $(x,a)=(x,y)\in\bp_*$) $x\neq a$, and so $(x,z)\neq(a,a)$. This proves that~$\bp_*$ is closed.

Let $x\tre y\tre z$ with $(x,z)\in\bp_*$, and suppose by way of
contradiction that $(x,y)\notin\bp_*$ and $(y,z)\notin\bp_*$. As~$\bp$
is open, either $(x,y)\in\bp$ or $(y,z)\in\bp$, thus either
$(x,y)=(a,a)$ or $(y,z)=(a,a)$. In the first case
$(y,z)=(x,z)\in\bp_*$, and in the second case $(x,y)=(x,z)\in\bp_*$,
\contr\ in both cases. This proves that~$\bp_*$ is open.

Finally let $\bu\subseteq\bp$ be open not contained in~$\bp_*$, so
$(a,a)\in\bu$. Let $(x,y)\in\bp$, we must prove that $(x,y)\in\bu$. If
$(x,y)=(a,a)$ this is trivial. Suppose that $x=a$ and $y\in
U\setminus\set{a}$. Then $a\tre y\tre a$, but $(y,a)\notin\bu$
(because $(y,a)\notin\bp$), $(a,a)\in\bu$, and~$\bu$ is open, thus
$(x,y)=(a,y)\in\bu$, as desired. This completes the case $x=a$. The
case where $y=a$ and $x\in U^\cpl\setminus\set{a}$ is dealt with similarly. Now suppose that $x\in U^\cpl\setminus\set{a}$ and
$y\in U\setminus\set{a}$. As above, we prove that $(y,a)\notin\bu$ and
thus $(a,y)\in\bu$. Now $(a,x)\notin\bu$ (because $(a,x)\notin\bp$),
thus, as $a\tre x\tre y$, $(a,y)\in\bu$, and~$\bu$ is open, we get
$(x,y)\in\bu$, as desired.
\end{proof}

\begin{corollary}\label{C:abUClopJirr}
  Let $(a,b,U)\in\cF(\be)$.  The clopen set $\bp=\pji{a}{b}{U}$ is
  completely \jirr\ in the lattice~$\Reg(\be)$, and the
  element~$\bp_*$ constructed in the statement of
  Lemma~\textup{\ref{L:abUClopJirr}} is the lower cover of~$\bp$ in
  that lattice.
\end{corollary}

\begin{proof}
Let $\ba\subsetneqq\bp$ be a regular closed set. As $\tin(\ba)$ is open and properly contained in~$\bp$, it follows from Lemma~\ref{L:abUClopJirr} that $\tin(\ba)\subseteq\bp_*$. Hence, as~$\ba$ is regular closed and~$\bp_*$ is clopen, we get $\ba=\tcl\tin(\ba)\subseteq\bp_*$.
\end{proof}

\begin{lemma}\label{L:OpenUjirr}
Every open subset~$\bu$ of~$\be$ is the set-theoretic union of all its subsets of the form $\pji{a}{b}{U}$, where $(a,b,U)\in\cF(\be)$. In particular, every open subset of~$\be$ is a union of clopen sets.
\end{lemma}

\begin{proof}
Let $(a,b)\in\bu$, we must find~$U$ such that $(a,b,U)\in\cF(\be)$ and $\pji{a}{b}{U}\subseteq\bu$.

Suppose first that $a\neq b$ and set
 \[
 U=\setm{x\in\cce{a,b}}{x\neq a
 \text{ and }(a,x)\in\bu}\,.
 \]
It is trivial that $a\notin U$ and $b\in U$. Let $(x,y)\in\pji{a}{b}{U}$ (so $x\notin U$ and $y\in U$), we must prove that $(x,y)\in\bu$. If $x=a$ then this is obvious  (because $y\in U$). Now suppose that $x\neq a$. In that case, from $x\notin U$ it follows that $(a,x)\notin\bu$. As $(a,y)\in\bu$ (because $y\in U$), $a\tre x\tre y$, and~$\bu$ is open, we get $(x,y)\in\bu$, as desired.

{}From now on suppose that $a=b$. We set
 \[
 U=\setm{x\in\cce{a}}{(a,x)\in\bu}\,.
 \]
 Observe that $a\in U$, so $\set{a}\cup U=U$.  Let
 $(x,y)\in\pji{a}{a}{U}$, we must prove that $(x,y)\in\bu$. If $x=a$,
 then, as $y\in U$, we get $(x,y)=(a,y)\in\bu$. Hence we may suppose
 from now on that $x\neq a$; it follows that $x\in U^\cpl$ (as $(x,y)\in\pji{a}{a}{U}$) and therefore $(a,x)\notin\bu$. As $y\in U$, we get $(a,y)\in\bu$.  As $a\tre x\tre y$ and~$\bu$ is open, it follows again that $(x,y)\in\bu$, as desired.
\end{proof}

\begin{corollary}\label{C:MacNeille}
  \label{blue:C:MD}
  The lattice~$\Reg(\be)$ is, up to 
  isomorphism, the Dedekind-MacNeille completion of the poset $\Clop(\be)$.
  In particular, every completely \jirr\ element of~$\Reg(\be)$ is clopen.
\end{corollary}

\begin{proof}
  Let $\ba$ be regular closed, so that $\ba = \tcl(\bb)$, with $\bb$
  open. Write $\bb$ as a union of clopen sets, $\bb = \bigcup_{i \in
    I} \bc_{i}$. Applying the closure operator~$\tcl$ to both sides of the
  equation, we obtain the equality $\ba = \bigvee_{i \in I} \bc_{i}$ in $\Reg(\be)$.
  
  Thus every element of $\Reg(\be)$ is a join of elements from
  $\Clop(\be)$; by duality, every element of $\Reg(\be)$ is a meet of
  elements from $\Clop(\be)$.  It immediately follows, see Davey and
  Priestley \cite[Theorem~7.41]{DP02}, that $\Reg(\be)$ is the
  Dedekind-MacNeille completion of the poset $\Clop(\be)$.
  
  Suppose next that $\ba$ is a completely \jirr\ element of
  $\Reg(\be)$; since $\Reg(\be)$ is join-generated by $\Clop(\be)$,
  we can write $\ba$ as join of clopen sets, $\ba = \bigvee_{i \in I}
  \bc_{i}$. As $\ba$ is completely \jirr, it follows that
  $\ba = \bc_{i}$ for some $i \in I$, thus $\ba$ is clopen.
\end{proof}

Lemma~\ref{L:OpenUjirr} makes it possible to extend Pouzet~\emph{et
  al.} \cite[Lemma~11]{PRRZ}, from permutohedra on posets, to lattices
of regular closed subsets of transitive relations. This result also
refines the equivalence (ii)$\Leftrightarrow$(iii) of
Theorem~\ref{T:SqFree}.

\begin{corollary}
  \label{blue:C:OpenUjirr}
  The following statements hold, for any family $\vecm{\ba_i}{i\in I}$
  of clopen subsets of~$\be$:
  \begin{enumerate}
  \item The set $\setm{\ba_i}{i\in I}$ has a meet in~$\Clop(\be)$
    if{f} $\tin\pI{\bigcap\vecm{\ba_i}{i\in I}}$ is clopen, and then
    the two sets are equal.
    
  \item The set $\setm{\ba_i}{i\in I}$ has a join in~$\Clop(\be)$
    if{f} $\tcl\pI{\bigcup\vecm{\ba_i}{i\in I}}$ is clopen, and then
    the two sets are equal.
  \end{enumerate}
\end{corollary}
\begin{proof}
  A simple application of the involution $\bx\mapsto\be\setminus\bx$
  reduces~(ii) to~(i). On the way to proving~(i), we set
  $\bu=\tin\pI{\bigcap\vecm{\ba_i}{i\in I}}$; so~$\bu$ is open.
  
  It is trivial that if~$\bu$ is clopen, then it is the meet of
  $\setm{\ba_i}{i\in I}$ in $\Clop(\be)$. Conversely, suppose that
  $\setm{\ba_i}{i\in I}$ has a meet~$\ba$ in $\Clop(\be)$. It is
  obvious that $\ba\subseteq\bu$. Let $(x,y)\in\bu$. By
  Lemma~\ref{L:OpenUjirr}, there exists $\bb\subseteq\bu$ clopen such
  that $(x,y)\in\bb$. It follows from the definition of~$\ba$ that
  $\bb\subseteq\ba$, thus $(x,y)\in\ba$.  Therefore, $\bu=\ba$ is
  clopen.
\end{proof}

Notice that Corollary~\ref{blue:C:OpenUjirr} can also be derived from
Corollary~\ref{blue:C:MD}: if the inclusion of $\Clop(\be)$ into
$\Reg(\be)$ is, up to isomorphism, the Dedekind-MacNeille completion
of $\Clop(\be)$, then this inclusion preserves existing joins and
meets. Thus, for example, suppose that $\bw =
\bigwedge_{\Clop(\be)}\vecm{\ba_i}{i\in I}$ exists; then $\bw =
\bigwedge_{\Reg(\be)} \vecm{\ba_i}{i\in I}$, so that
\begin{align*}
  \bw = {\bigwedge\!{}_{\Reg(\be)}}
  \vecm{\ba_i}{i\in I} 
  & = \tcl \tin\pI{\bigcap\vecm{\ba_i}{i\in I}} 
  \supseteq \tin\pI{\bigcap\vecm{\ba_i}{i\in I}}\,.
\end{align*}
As $\bw \subseteq \tin\pI{\bigcap\vecm{\ba_i}{i\in I}}$ follows from
the openness of~$\bw$ together with $\bw \subseteq \ba_i$ for each
${i\in I}$, we get $\bw = \tin\pI{\bigcap\vecm{\ba_i}{i\in I}}$,
showing that $\tin\pI{\bigcap\vecm{\ba_i}{i\in I}}$ is closed.

%% A proof if we want ti use the MacNeille completion only.
%% \begin{proof}
%%   A simple application of the involution $\bx\mapsto\be\setminus\bx$
%%   reduces~(ii) to~(i). On the way to proving~(i), we set
%%   $\bu=\tin\pI{\bigcap\vecm{\ba_i}{i\in I}}$; so~$\bu$ is open.
  
%%   It is trivial that if~$\bu$ is clopen, then it is the meet of $\setm{\ba_i}{i\in I}$ in $\Clop(\be)$.
  
%%   Conversely, argue as follows. As $\Reg(\be)$ is, up to isomorphism, the
%%   Dedekind-MacNeille $\Clop(\be)$, the inclusion of $\Clop(\be)$ into
%%   $\Reg(\be)$ preserves existing joins. 
%%   Suppose that $\bw = \bigwedge_{\Clop(\be)}\vecm{\ba_i}{i\in I}$
%%   exists; then $\bw = \bigwedge_{\Reg(\be)} \vecm{\ba_i}{i\in I}$, so
%%   that
%%   \begin{align*}
%%     \bw = {\bigwedge\!{}_{\Reg(\be)}}
%%     \vecm{\ba_i}{i\in I} 
%%     & = \tcl \tin\pI{\bigcap\vecm{\ba_i}{i\in I}} 
%%     \supseteq \tin\pI{\bigcap\vecm{\ba_i}{i\in I}} = \bu\,.
%%   \end{align*}
%%   As $\bw \subseteq \tin\pI{\bigcap\vecm{\ba_i}{i\in I}} =\bu$ follows from
%%   the openness of~$\bw$ together with $\bw \subseteq \ba_i$ for each
%%   ${i\in I}$, we get $\bw = \bu$,
%%   showing that $\bu$ is closed.
%% \end{proof}

\begin{theorem}
\label{T:JirrRege}
\label{T:Regespatial}
The completely \jirr\ elements of $\Reg(\be)$ are exactly the elements
$\pji{a}{b}{U}$, where $(a,b,U)\in\cF(\be)$. Furthermore, the lattice
$\Reg(\be)$ is spatial.
\end{theorem}

\begin{proof}
  Let $\ba\in\Reg(\be)$. As $\tin(\ba)$ is open, it follows from
  Lemma~\ref{L:OpenUjirr} that we can write $\tin(\ba)=\bigcup_{i\in
    I}\pji{a_i}{b_i}{U_i}$, for a family $\vecm{(a_i,b_i,U_i)}{i\in
    I}$ of elements of~$\cF(\be)$. As the elements
  $\pji{a_i}{b_i}{U_i}$ are all clopen (thus regular closed) and~$\ba$
  is regular closed, it follows that
 \[
 \ba=\tcl\tin(\ba)=
 \bigvee_{i\in I}\pji{a_i}{b_i}{U_i}\quad
 \text{in }\Reg(\be)\,.
 \]
In particular, if~$\ba$ is completely \jirr, then it must be one of the $\pji{a_i}{b_i}{U_i}$. Conversely, by Corollary~\ref{C:abUClopJirr}, every element of the form  $\pji{a}{b}{U}$ is completely \jirr\ in $\Reg(\be)$.
\end{proof}

\section{The arrow relations between clopen sets}\label{S:JoinDep}

Lemma~\ref{L:Arr2D} makes it possible to express the join-dependency relation on a finite lattice in terms of the arrow relations~$\nearrow$ and~$\searrow$. Now let~$\be$ be a transitive relation on a finite set~$E$. By using the dual automorphism $\bx\mapsto\orth{\bx}$ (cf. Lemma~\ref{L:BasicOrth}), $\orth{\bx}\searrow\by$ if{f} $\bx\nearrow\orth{\by}$, for all $\bx,\by\in\Reg(\be)$; hence statements involving~$\searrow$ can always be expressed in terms of~$\nearrow$. Furthermore, the \mirr\ elements of $\Reg(\be)$ are exactly the elements of the form $\orth{\bp}$, where~$\bp$ is a \jirr\ element of $\Reg(\be)$. As every such~$\bp$ is clopen (cf. Theorem~\ref{T:JirrRege}), we get $\orth{\bp}=\be\setminus\bp$. Therefore, we obtain the following lemma.

\begin{lemma}\label{L:ArRelJirr}
$\bp\nearrow\orth{\bq}$ if{f} $\bp\cap\bq\neq\es$ and $\bp\cap\bq_*=\es$, for all \jirr\ clopen sets~$\bp$ and~$\bq$.
\end{lemma}

{}From Lemma~\ref{L:Ar2IntDecr} to Lemma~\ref{L:Araneqbcneqd}, we shall fix $(a,b,U),(c,d,V)\in\cF(\be)$. Further, we shall set
$\bp=\pji{a}{b}{U}$, $\bq=\pji{c}{d}{V}$, $U^\cpl=\cce{a,b}\setminus
U$, and $V^\cpl=\cce{c,d}\setminus V$. Notice that if $\bp\nearrow\orth{\bq}$, then 
  \[
  \bp \subseteq(\bq_{*})^{\cpl}
  \subseteq
  \pI{\bq \setminus(\cce{c}\times\cce{d})}^{\cpl}
  = \bq^{\cpl} \cup (\cce{c}\times\cce{d})\,,
  \]
so that $\es \neq \bp \cap \bq \subseteq \cce{c}\times\cce{d}$.

\begin{lemma}
  \label{L:Ar2IntDecr}
  $\bp\nearrow\orth{\bq}$ implies that $\cce{c,d}\subseteq\cce{a,b}$.
\end{lemma}

\begin{proof}
  It follows from Lemma~\ref{L:abUClopJirr}, and more particularly
  from~\eqref{Eq:Defp*}, that~$\bq_*$ contains
  $\bq\setminus(\cce{c}\times\cce{d})$. By Lemma~\ref{L:ArRelJirr}, it
  follows that $\bp\nearrow\orth{\bq}$ implies that $\bp\cap\bq$ is
  nonempty and contained in $\cce{c}\times\cce{d}$. Pick any element $(x,y)\in \bp\cap\bq$. Then $a\utre x\tre y\utre b$ on the one hand,
  while $x\eqe c$ and $y\eqe d$ on the other hand. The desired conclusion follows.
\end{proof}

\begin{lemma}\label{L:Arc=d}
Suppose that $c=d$. Then $\bp\nearrow\orth{\bq}$ if{f} $a=b=c=d$, $U\cap V\subseteq\set{a}$, and $U^\cpl\cap V^\cpl\subseteq\set{a}$.
\end{lemma}

\begin{proof}
Suppose first that $\bp\nearrow\orth{\bq}$. {}From $\bq_*=\bq\setminus\set{(c,c)}$ (cf.~\eqref{Eq:Defp*}) it follows that
 \begin{equation}\label{Eq:bpcapbqcc}
 \bp\cap\bq=\set{(c,c)}\,,
 \end{equation}
and thus $(c,c)\in\pji{a}{b}{U}$, which rules out $a\neq b$ (cf. Definition~\ref{D:abU}). Hence $a=b$ and~$c$ belongs to $(\set{a}\cup U^\cpl)\cap(\set{a}\cup U)=\set{a}$, so $a=c$. Now let $x\in U^\cpl\cap V^\cpl$. Then $(x,a)$ belongs to $\bp\cap\bq$, thus, by~\eqref{Eq:bpcapbqcc}, $x=a$. The proof of the containment $U\cap V\subseteq\set{a}$ is similar.

Conversely, suppose that $a=b=c=d$, $U\cap V\subseteq\set{a}$, and $U^\cpl\cap V^\cpl\subseteq\set{a}$. Then any $(x,y)\in\bp\cap\bq$ satisfies $x\in U^\cpl\cap V^\cpl$, thus $x=a$. Likewise, $y=a$.
\end{proof}

As a noteworthy consequence, we obtain that if~$\bq$ is a clepsydra, then there is no~$\bp$ such that $\bp\Dr\bq$:

\begin{corollary}\label{C:Arc=d}
If~$c=d$, then the relation $\bp\Dr\bq$ does not hold.
\end{corollary}

\begin{proof}
Suppose that $\bp\Dr\bq$, so that there exists a \jirr\ element~$\br$ such that $\bp\nearrow\orth{\br}$ and $\br\nearrow\orth{\bq}$. By
  Lemma~\ref{L:Arc=d}, $a=b=c=d$ and there exists $W\subseteq\cce{a}$
  such that $\br=\pji{a}{a}{W}$ and all the sets $U\cap W$,
  $U^\cpl\cap W^\cpl$, $V\cap W$, and $V^\cpl\cap W^\cpl$ are
  contained in~$\set{a}$. It follows that $U\cup\set{a}=V\cup\set{a}$,
  thus $\bp=\bq$, \contr.
\end{proof}

\begin{lemma}\label{L:Arcneqd}
Suppose that $a=b$ and $c\neq d$. Then $\bp\nearrow\orth{\bq}$ if{f} $a\eqe c\eqe d$, $(\set{a}\cup U)\cap V\neq\es$, and $(\set{a}\cup U^\cpl)\cap V^\cpl\neq\es$.
\end{lemma}

\begin{proof}
Suppose first that $\bp\nearrow\orth{\bq}$. It follows from Lemma~\ref{L:Ar2IntDecr} that $a\eqe c\eqe d$. Any element $(u,v)\in\bp\cap\bq$ satisfies that $u\in(\set{a}\cup U^\cpl)\cap V^\cpl$ and $v\in(\set{a}\cup U)\cap V$. Conversely, if $a\eqe c\eqe d$, $u\in(\set{a}\cup U^\cpl)\cap V^\cpl$, and $v\in(\set{a}\cup U)\cap V$, then $(u,v)\in\bp\cap\bq$. {}From $c\neq d$, $c\eqe d$, and $\bq\subseteq\cce{c,d}\times\cce{c,d}$ it follows that $\bq_*=\bq\setminus(\cce{c}\times\cce{d})=\es$; whence $\bp\cap\bq_*=\es$.
\end{proof}

  \begin{lemma}\label{L:Araneqbcneqd}
  Suppose that $a\neq b$ and $c\neq d$. Then $\bp\nearrow\orth{\bq}$ if{f} $\cce{c,d}\subseteq\cce{a,b}$ and $\es\neq \be \cap\pI{(U^\cpl\cap V^\cpl) \times (U\cap V)}\subseteq\cce{c}\times \cce{d}$.
  \end{lemma}
  
\begin{proof}
  Suppose first that $\bp\nearrow\orth{\bq}$. It follows from
  Lemma~\ref{L:Ar2IntDecr} that $\cce{c,d}\subseteq\cce{a,b}$. Any
  element $(u,v)\in\bp\cap\bq$ belongs to $\be \cap \pI{(U^\cpl\cap V^\cpl)\times(U\cap V)}$, so that this set is nonempty. Moreover, any element $(u,v)$ of this set belongs to $\bq\setminus\bq_*$, thus
  $u\eqe c$ and $v\eqe d$.
  Conversely, suppose that $\cce{c,d}\subseteq\cce{a,b}$ and $\es\neq
  \be \cap\pI{(U^\cpl\cap V^\cpl) \times (U\cap
    V)}\subseteq\cce{c}\times \cce{d}$.  Observing that $\bp\cap\bq =
  \be \cap\pI{(U^\cpl\cap V^\cpl) \times (U\cap V)}$, it follows that
  $\bp\cap\bq$ is both nonempty and contained in
  $\cce{c}\times\cce{d}$; the latter condition implies that
  $\bp\cap\bq_* = \es$. We get therefore $\bp\nearrow\orth{\bq}$.
\end{proof}

\begin{corollary}\label{C:Araneqbcneqd}
Suppose that~$\be$ is antisymmetric, $a\neq b$, and $c\neq d$. Then $\bp\nearrow\orth{\bq}$ if{f} $(c,d)\in\bp$ and $V=(\oce{c,d}\setminus U)\cup\set{d}$.
\end{corollary}

\section{Bounded lattices of regular closed sets}\label{S:Bounded}

Let~$\be$ be a transitive relation on a set~$E$. Suppose, until the statement of Proposition~\ref{P:CharDrel}, that~$\be$ is \emph{antisymmetric} (i.e., the preordering~$\utre$ is an ordering), then some information can be added to the results of Section~\ref{S:JoinDep}. First of all, the clepsydras (cf. Definition~\ref{D:abU}) are exactly the singletons $\set{(a,a)}$, where $a\tre a$. On the other hand, if $a\neq b$, then $\pji{a}{b}{U}$ determines both the ordered pair $(a,b)$ and the set~$U$. (Recall that a clepsydra is never bipartite, so it cannot be of the form $\pji{a}{b}{U}$ with $a\neq b$.)

Let us focus for a while on arrow relations involving bipartite \jirr\ clopen sets. We set
 \[
 \Res{U}{c,d}=(U\cap\oce{c,d})\cup\set{d}\,,\quad\text{for all }(c,d)\in\be
 \text{ and all }U\subseteq E\,.
 \]
 Observe that $(c,d,\Res{U}{c,d})\in\cF(\be)$. The following lemma is an
 easy consequence of the antisymmetry of~$\be$ together with
 Lemma~\ref{L:Araneqbcneqd}.

\begin{lemma}\label{L:Ar2tilde}
  Let $(a,b,U),(c,d,V)\in\cF(\be)$ with $a\neq b$ and $c\neq d$, and
  set $\widetilde{U}=\Res{(E\setminus U)}{a,b}$. Then
  $\pji{a}{b}{U}\nearrow\orth{\pji{c}{d}{V}}$ if{f}
  $\cce{c,d}\subseteq\cce{a,b}$ and $V=\Res{\widetilde{U}}{c,d}$.
\end{lemma}

This yields, in the finite case, a characterization of the join-dependency relation on the \jirr\ clopen sets.

\begin{proposition}\label{P:CharDrel}
  Suppose that~$E$ is finite, $\be$ is antisymmetric, and let
  $(a_0,b_0,U_0)$, $(a_1,b_1,U_1)\in\cF(\be)$ with $a_0\neq b_0$ and
  $a_1\neq b_1$. Set $\bp_i=\pji{a_i}{b_i}{U_i}$ for $i<2$. Then
  $\bp_0\Dr\bp_1$ in the lattice $\Reg(\be)$ if{f}
  $\cce{a_1,b_1}\subsetneqq\cce{a_0,b_0}$ and $U_1=\Res{U_0}{a_1,b_1}$.
\end{proposition}

\begin{proof}
  Suppose first that $\bp_0\Dr\bp_1$. By Theorem~\ref{T:JirrRege},
  Lemma~\ref{L:Arr2D}, and the observations at the beginning of
  Section~\ref{S:JoinDep}, there exists $(c,d,V)\in\cF(\be)$ such
  that, setting $\bq=\pji{c}{d}{V}$, the relations
  $\bp_0\nearrow\orth{\bq}$ and $\bq\nearrow\orth{\bp}_1$ both
  hold. It follows from Corollary~\ref{C:Arc=d} that $c\neq d$. By
  Lemma~\ref{L:Ar2tilde},
  $\cce{a_1,b_1}\subseteq\cce{c,d}\subseteq\cce{a_0,b_0}$ and, setting
  $\widetilde{U}_0=\Res{(E\setminus U_0)}{a_0,b_0}$ and
  $\widetilde{V}=\Res{(E\setminus V)}{c,d}$,
  $V=\Res{\widetilde{U}_0}{c,d}$ and
  $U_1=\Res{\widetilde{V}}{a_1,b_1}$; clearly $\widetilde{V} =
  \Res{U_{0}}{c,d}$, whence $U_1=\Res{U_0}{a_1,b_1}$. Since
  $\bp_0\neq\bp_1$, it follows that
  $\cce{a_1,b_1}\subsetneqq\cce{a_0,b_0}$.

Conversely, suppose that $\cce{a_1,b_1}\subsetneqq\cce{a_0,b_0}$ and
$U_1=\Res{U_0}{a_1,b_1}$. In particular, $\bp_0\neq\bp_1$. Set
$U_0^\cpl=\cce{a_0,b_0}\setminus U_0$. We shall separate cases,
according to whether or not~$a_1$, $b_1$ belong to~$U_0$. In each of
those cases, we shall define a certain \jirr\ element
$\bq=\pji{c}{d}{V}$ of~$\Reg(\be)$, with $a_0\utre c\tre d\utre a_1$
and $V=\Res{U_0^\cpl}{c,d}$, so that only~$c$ and~$d$ will need to be
specified. Each of the desired arrow relations will be inferred with
the help of Corollary~\ref{C:Araneqbcneqd}.

\subsubsection*{Case 1} $a_1\in\set{a_0}\cup U_0^\cpl$ and $b_1\in\set{b_0}\cup U_0$. We set $c=a_1$ and $d=b_1$. Then $\bp_0\nearrow\orth{\bq}$ (because $(a_1,b_1)\in\bp_0$) and $\bq\nearrow\orth{\bp}_1$ (because $(a_1,b_1)\in\bq$).

\subsubsection*{Case 2} $a_1\in\set{a_0}\cup U_0$ and $b_1\in\set{b_0}\cup U_0$. We set $c=a_0$ and $d=b_1$. Then $\bp_0\nearrow\orth{\bq}$ (because $(a_0,b_1)\in\bp_0$) and $\bq\nearrow\orth{\bp}_1$ (because $(a_1,b_1)\in\bq$).

\subsubsection*{Case 3} $a_1\in\set{a_0}\cup U_0^\cpl$ and $b_1\in\set{b_0}\cup U_0^\cpl$. We set $c=a_1$ and $d=b_0$. Then $\bp_0\nearrow\orth{\bq}$ (because $(a_1,b_0)\in\bp_0$) and $\bq\nearrow\orth{\bp}_1$ (because $(a_1,b_1)\in\bq$).

\subsubsection*{Case 4} $a_1\in\set{a_0}\cup U_0$ and $b_1\in\set{b_0}\cup U_0^\cpl$. We set $c=a_0$ and $d=b_0$. Then $\bp_0\nearrow\orth{\bq}$ (because $(a_0,b_0)\in\bp_0$) and $\bq\nearrow\orth{\bp}_1$ (because $(a_1,b_1)\in\bq$).

\medskip
In each of those cases, $\bp_0\nearrow\orth{\bq}$ and $\bq\nearrow\orth{\bp}_1$, hence, as $\bp_0\neq\bp_1$ and by Lemma~\ref{L:Arr2D}, $\bp_0\Dr\bp_1$.
\end{proof}

By using the standard description of the congruence lattice of a
finite lattice \emph{via} the join-dependency relation (cf. Freese,
Je\v{z}ek, and Nation \cite[Section~II.3]{FJN}),
Proposition~\ref{P:CharDrel} makes it possible to give a complete description of the congruence lattice of $\Reg(\be)$ in case~$\be$ is antisymmetric. Congruence lattices of permutohedra were originally described in
Duquenne and Cherfouh~\cite[Section~4]{DuCh94}. An implicit description of congruences of permutohedra \emph{via} the join-dependency relation appears in Santocanale~\cite{San07}; in that paper, similar results were established for multinomial lattices.

By Lemma~\ref{L:Arcneqd} together with the antisymmetry of~$\be$, if $\bp\nearrow\orth{\bq}$ and~$\bp$ is a clepsydra, then so is~$\bq$. Hence, by Corollary~\ref{C:Arc=d}, $\bp\Dr\bq$ implies (still in the antisymmetric case) that neither~$\bp$ nor~$\bq$ is a clepsydra. By Proposition~\ref{P:CharDrel}, we thus obtain the following result.

\begin{corollary}\label{C:DrelStrOrd}
Let $\be$ be an antisymmetric, transitive relation on a finite set. Then the join-dependency relation on the \jirr\ elements of~$\Reg(\be)$ is a strict ordering.
\end{corollary}

\begin{example}\label{Ex:DNotTrans}
The transitivity of the~$\Dr$ relation, holding on the \jirr\ elements of~$\Reg(\be)$ for any antisymmetric transitive relation~$\be$, is quite a special property. It does not hold in all finite bounded homomorphic images of free lattices, as shows the lattice~$\sL_9$ (following the notation of Jipsen and Rose~\cite{JiRo}) represented on the left hand side of Figure~\ref{Fig:L9}. The \jirr\ elements marked there by doubled circles satisfy $p\Dr q$ and $q\Dr r$ but not $p\Dr r$.

\begin{figure}[htb]
\includegraphics{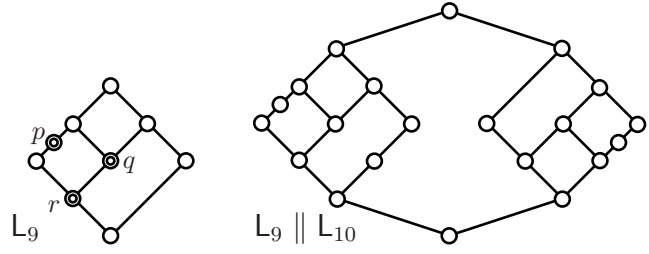}
\caption{Bounded lattices with non-transitive join-de\-pen\-den\-cy relation}
\label{Fig:L9}
\end{figure}

The lattice~$\sL_9$ is not orthocomplemented, but its parallel sum
with its dual lattice~$\sL_{10}$, denoted there by
$\sL_9\parallel\sL_{10}$, is orthocomplemented. As~$\sL_9$, the parallel sum is bounded and has non-transitive~$\Dr$ relation.

In the non-bounded case, the reflexive closure of the~$\Dr$ relation on~$\Reg(\be)$ may not be transitive. This is witnessed by the lattice $\Bip(3)=\Reg([3]\times[3])$, see Lemmas~\ref{L:DRelonBipart} and~\ref{L:TCDBipBip}.
\end{example}

\begin{definition}\label{D:Orthei}
A family $\vecm{\be_i}{i\in I}$ of \emph{pairwise disjoint} transitive relations is \emph{orthogonal} if there are no distinct $i,j\in I$ and no $p,q,r$ such that $p\neq q$, $q\neq r$, $(p,q)\in\be_i$, and $(q,r)\in\be_j$.
\end{definition}

In particular, if $\vecm{\be_i}{i\in I}$ is orthogonal, then $\bigcup_{i\in I}\be_i$ is itself a transitive relation.

\begin{proposition}\label{P:Regsumei}
The following statements hold, for any orthogonal family\linebreak $\vecm{\be_i}{i\in I}$ of transitive relations:
\begin{enumerate}
\item A subset~$\bx$ of~$\be$ is closed \pup{resp., open in~$\be$} if{f} $\bx\cap\be_i$ is closed \pup{resp., open in~$\be_i$} for each $i\in I$.

\item $\Reg(\be)\cong\prod_{i\in I}\Reg(\be_i)$, \emph{via} an isomorphism that carries $\Clop(\be)$ onto $\prod_{i\in I}\Clop(\be_i)$.

\end{enumerate}

\end{proposition}

\begin{proof}
The proof of~(i) is a straightforward exercise. For~(ii), we define $\varphi(\bx)=\vecm{\bx\cap\be_i}{i\in I}$ whenever $\bx\subseteq\be$, and $\psi\vecm{\bx_i}{i\in I}=\bigcup_{i\in I}\bx_i$ whenever all $\bx_i\subseteq\be_i$. By using~(i), it is straightforward (although somewhat tedious) to verify that~$\varphi$ and~$\psi$ restrict to mutually inverse isomorphisms between $\Reg(\be)$ and $\prod_{i\in I}\Reg(\be_i)$, and also between $\Clop(\be)$ and $\prod_{i\in I}\Clop(\be_i)$.
\end{proof}

Set $\Delta_A=\setm{(x,x)}{x\in A}$, for every set~$A$. By applying Proposition~\ref{P:Regsumei} to the $2$-element family $(\be,\Delta_A)$, we obtain the following result, which shows that $\Reg(\be\cup\Delta_A)$ is the product of~$\Reg(\be)$ by a powerset lattice.

\begin{corollary}\label{C:Regsumei}
Let~$\be$ be a transitive relation and let~$A$ be a set with $\be\cap\Delta_A=\es$. Then $\Reg(\be\cup\Delta_A)\cong\Reg(\be)\times(\Pow A)$ and $\Clop(\be\cup\Delta_A)\cong \Clop(\be)\times(\Pow A)$.
\end{corollary}

Duquenne and Cherfouh \cite[Theorem~3]{DuCh94} and Le Conte de Poly-Barbut \cite[Lemme~9]{Poly94} proved that every permutohedron is semidistributive (in the latter paper the result was extended to all \emph{Coxeter lattices}). This result was improved by Caspard~\cite{Casp00}, who proved that every permutohedron is a bounded homomorphic image of a free lattice; and later, by Caspard, Le Conte de Poly-Barbut, and Morvan~\cite{CLM04}, who extended this result to all finite Coxeter groups. Our next result shows exactly to which transitive (not necessarily antisymmetric) relations those results can be extended.

\begin{theorem}\label{T:RegeSD}
The following are equivalent, for any transitive relation~$\be$ on a finite set~$E$:
\begin{enumerate}
\item The lattice~$\Reg(\be)$ is a bounded homomorphic image of a free lattice.

\item The lattice~$\Reg(\be)$ is semidistributive.

\item The lattice~$\Reg(\be)$ is pseudocomplemented.

\item Every connected component of the preordering~$\utre$ either is antisymmetric or has the form $\set{a,b}$ with $a\neq b$ while $(a,b)\in\be$ and $(b,a)\in\be$.
\end{enumerate}
\end{theorem}

\begin{proof}
(i)$\Rightarrow$(ii) is well-known, see for example Freese, Je\v{z}ek, and Nation \cite[Theorem~2.20]{FJN}.

(ii)$\Rightarrow$(iii) is trivial.

(iii)$\Rightarrow$(iv) follows immediately from Proposition~\ref{P:NonPseudo}.

(iv)$\Rightarrow$(i). Denote by $\setm{E_i}{i<n}$ the set of all connected components of~$\utre$ and set $\be_i=\be\cap(E_i\times E_i)$ for each $i<n$. By Proposition~\ref{P:Regsumei}, it suffices to consider the case where $\be=\be_i$ for some~$i$, that is, $\utre$ is connected. It suffices to prove that the join-dependency relation on~$E$ has no cycle (cf. Freese, Je\v{z}ek, and Nation \cite[Corollary~2.39]{FJN} together with the self-duality of $\Reg(\be)$). In the antisymmetric case, this follows from Corollary~\ref{C:DrelStrOrd}. If $E=\set{a,b}$ with $a\neq b$ while $(a,b)\in\be$ and $(b,a)\in\be$, then $\Reg(\be)$ is isomorphic to the lattice~$\Bip(2)$ of all bipartitions of a two-element set (cf. Section~\ref{S:RegCl}). This lattice is represented on the left hand side of Figure~\ref{Fig:Bip2}. On the right hand side of Figure~\ref{Fig:Bip2} we represent the~$\Dr$ relation on the \jirr\ elements of~$\Bip(2)$. The \jirr\ elements of~$\Bip(2)$ are denoted there by
 \begin{align*}
 \ba_0&=\set{(1,2)}\,,&
 \ba_1&=\set{(1,1),(1,2)}\,,&
 \ba_2&=\set{(2,2),(1,2)}\,,\\
 \bb_0&=\set{(2,1)}\,,&
 \bb_1&=\set{(1,1),(2,1)}\,,&
 \bb_2&=\set{(2,2),(2,1)}\,.
 \end{align*}
In particular, the~$\Dr$ relation on~$\Ji(\Bip(2))$ has no cycle. Hence, as~$\Bip(2)$ is self-dual, it is a bounded homomorphic image of a free lattice.
\end{proof}

\begin{corollary}\label{C:DrelStrOrd2}
Let~$\be$ be a finite transitive relation. If~$\Reg(\be)$ is semidistributive, then the join-dependency relation defines a strict ordering on the \jirr\ elements of $\Reg(\be)$.
\end{corollary}

\begin{proof}
By using Proposition~\ref{P:Regsumei}, together with the characterization~(iv) of semidistributivity of~$\Reg(\be)$ given in Theorem~\ref{T:RegeSD}, 
it is easy to reduce
the problem to the case where~$\be$ is either antisymmetric or a loop
$a\tre b\tre a$ with $a\neq b$. In the first case, the conclusion
follows from Corollary~\ref{C:DrelStrOrd}. In the second case, the
join-dependency relation is bipartite (see the right hand side of
Figure~\ref{Fig:Bip2}), thus transitive.
\end{proof}

The lattices~$\Bip(3)$ and~$\Bip(4)$ are represented on Figure~\ref{Fig:Bip34}; they have~$74$ and~$730$ elements, respectively.

\begin{figure}[htb]
\includegraphics{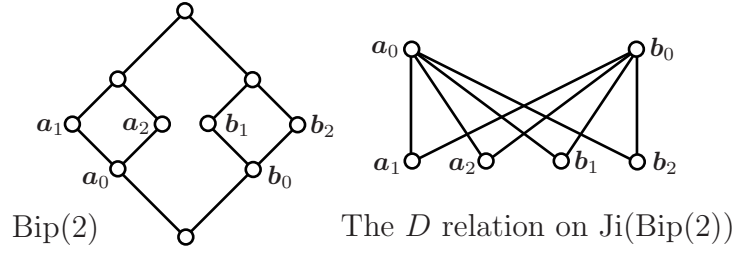}
\caption{The lattice $\Bip(2)$ of all bipartitions of $\set{1,2}$ and its~$\Dr$ relation}\label{Fig:Bip2}
\end{figure}

\begin{figure}[htb]
\centering
\includegraphics[scale=0.5]{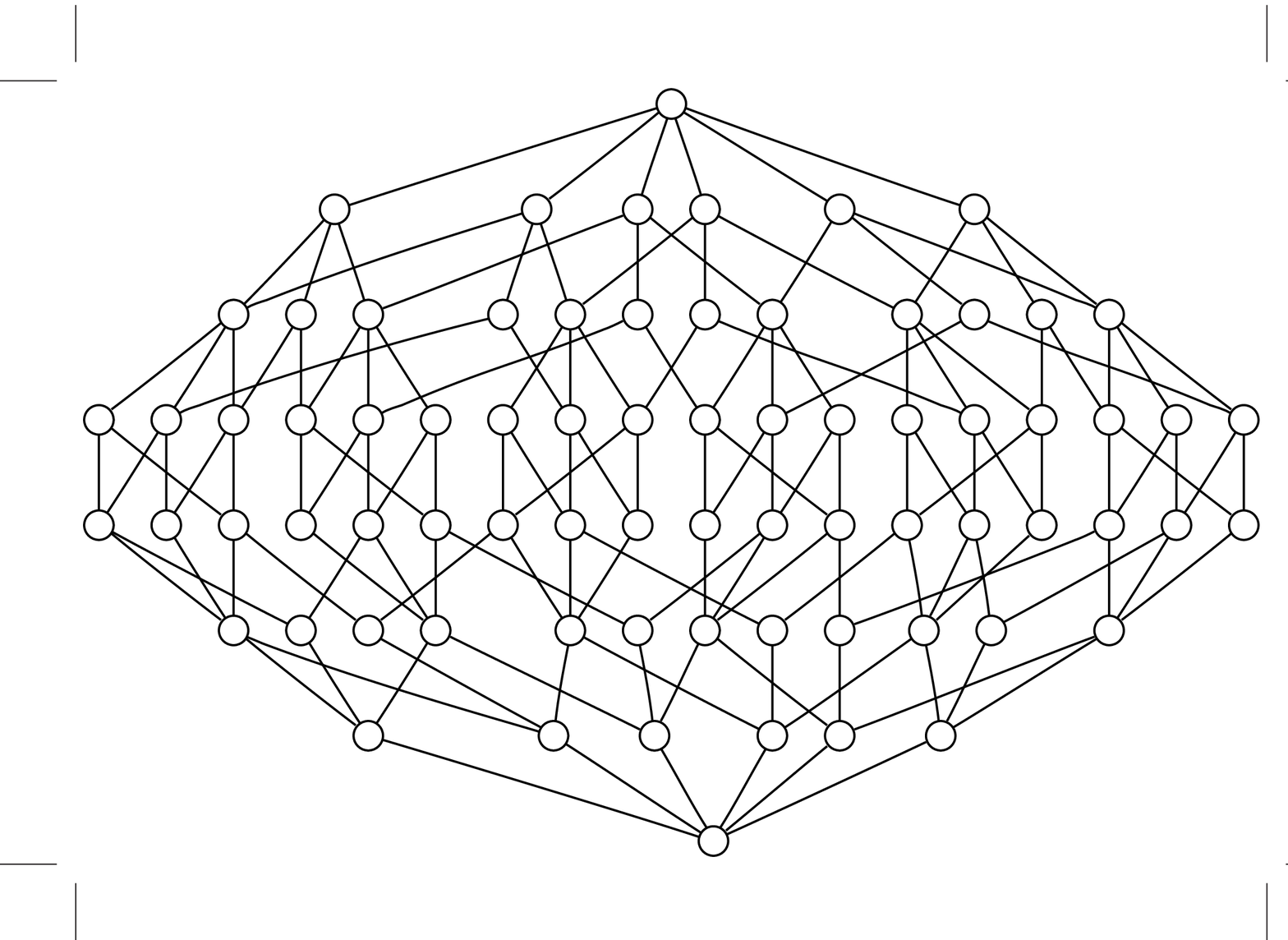}
\includegraphics[scale=0.08]{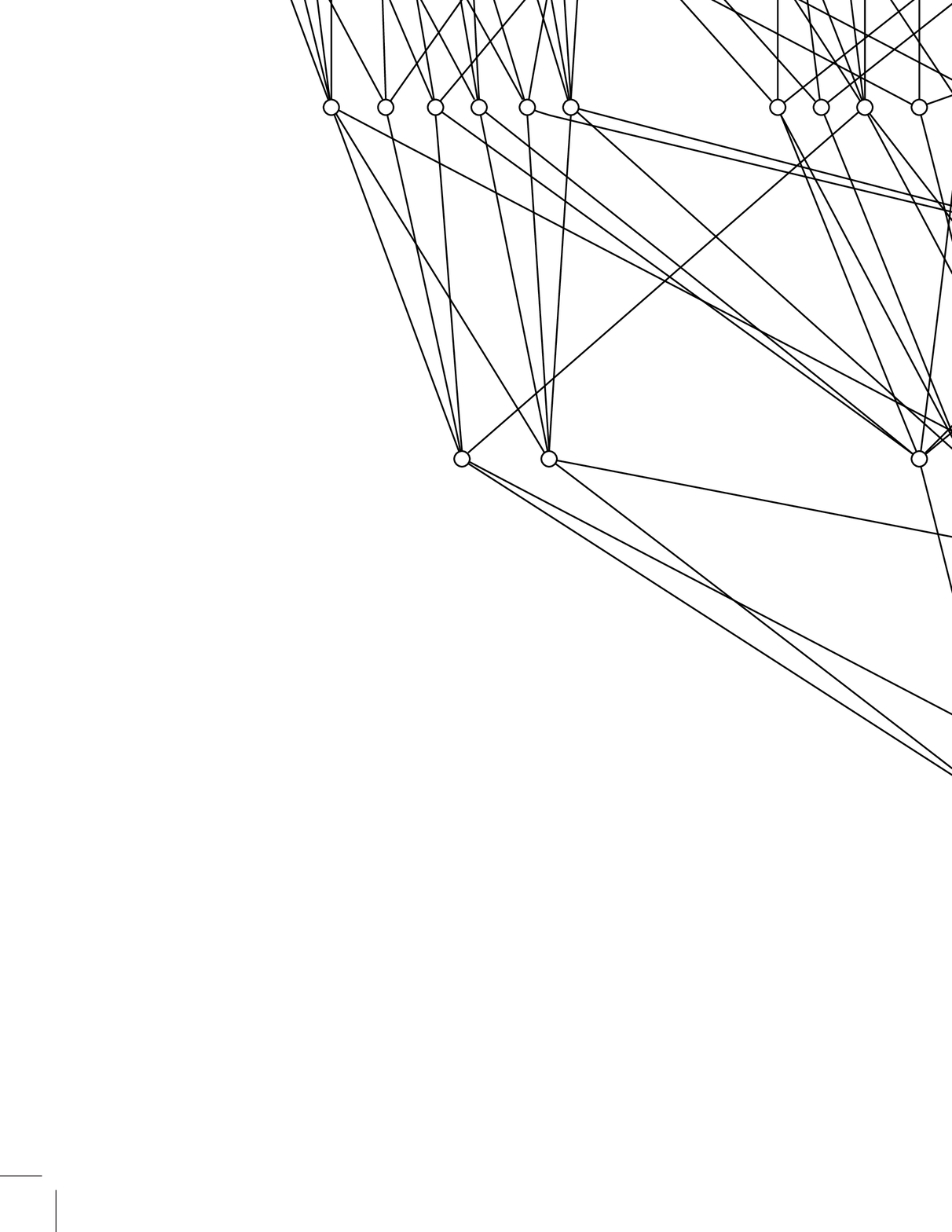}
\caption{The lattices $\Bip(3)$ and $\Bip(4)$}\label{Fig:Bip34}
\end{figure}

\begin{example}\label{Ex:Pomega+1}
The following example shows that none of the implications (iv)$\Rightarrow$(ii) and (iv)$\Rightarrow$(iii) of Theorem~\ref{T:RegeSD} can be extended to the infinite case. Define~$\be$ as the natural strict ordering on the ordinal $\go+1=\set{0,1,2,\dots}\cup\set{\go}$. As~$\be$ is obviously square-free, it follows from Theorem~\ref{T:SqFree} that $\Clop(\be)=\Reg(\be)$ is a lattice, namely the permutohedron~$\sP(\go+1)$ (cf. Example~\ref{Ex:PermPoset2}). It is straightforward to verify that the sets
 \begin{align*}
 \ba&=\setm{(2m,2n+1)}{m\leq n<\go}
 \cup\setm{(2m,\go)}{m<\go}\,,\\
 \bb&=\setm{(2m+1,2n+2)}{m\leq n<\go}
 \cup\setm{(2m+1,\go)}{m<\go}\,,\\
 \bc&=\setm{(m,\go)}{m<\go}\,.
 \end{align*}
are all clopen in~$\be$. Furthermore,
 \[
 \ba\cap\bc=\setm{(2m,\go)}{m<\go}
 \]
has empty interior, so $\ba\wedge\bc=\es$. Likewise, $\bb\wedge\bc=\es$. On the other hand, $\bc\subseteq\ba\cup\bb\subseteq\ba\vee\bb$ and $\bc\neq\es$. In particular,
 \[
 \ba\wedge\bc=\bb\wedge\bc=\es\text{ and }
 (\ba\vee\bb)\wedge\bc\neq\es\,.
 \]
Therefore, the lattice~$\sP(\go+1)$ is neither pseudocomplemented, nor semidistributive.

The result that every (finite) permutohedron~$\sP(n)$ is pseudocomplemented originates in Chameni-Nembua and Monjardet~\cite{ChaMon92}.
\end{example}

\section{Join-dependency and congruences of bipartition lattices}\label{S:JoinBip}

The full relation $\bi_n=[n]\times[n]$ is transitive, for any positive integer~$n$, and $\Reg(\bi_n)=\Clop(\bi_n)=\Bip(n)$ (cf. Section~\ref{S:RegCl}), the bipartition lattice of~$n$. Due to the existence of exactly one $\equiv_{\bi_n}$-class (namely the full set~$[n]$), the description of the \jirr\ elements of~$\Bip(n)$ obtained from Theorem~\ref{T:JirrRege} takes a particularly simple form.

\begin{lemma}\label{L:JirrBipn}
The \jirr\ elements of~$\Bip(n)$ are exactly the sets $\seq{U}=U^\cpl\times U$, for $U\in\Pow^*[n]$ \pup{those are the bipartite ones}, and $\seq{a,U}=\pji{a}{a}{U}$ for $a\in[n]$ and $U\in\Pow[n]$ \pup{those are the clepsydras}.
\end{lemma}

\begin{lemma}\label{L:DRelFromClep}
Let $(a,U)\in[n]\times\Pow[n]$ and let $V\in\Pow^*[n]$. Then $\seq{a,U}\Dr\seq{V}$ always holds \pup{within $\Bip(n)$}.
\end{lemma}

\begin{proof}
Suppose first that $(U,V)$ forms a partition of~$[n]$. It follows from Lemma~\ref{L:Arc=d} that
 $\seq{a,U}\nearrow\orth{\seq{a,V}}$; moreover, considering that both
 $V$ and $V^{\cpl}$ are nonempty, the relation
 $\seq{a,V}\nearrow\orth{\seq{V}}$  follows from Lemma~\ref{L:Arcneqd}. Hence $\seq{a,U}\Dr\seq{V}$.

Suppose from now on that $(U,V)$ does not form a partition of~$[n]$. We shall find $W\in\Pow[n]$ such that
 \begin{align}
 (\set{a}\cup U^\cpl)\cap W^\cpl&\neq\es\,,
 \label{Eq:aUcmeetsWc}\\
 V^\cpl\cap W^\cpl&\neq\es\,,
 \label{Eq:VcmeetsWc}\\
 (\set{a}\cup U)\cap W&\neq\es\,,
 \label{Eq:UmeetsW}\\
 V\cap W&\neq\es\,.\label{Eq:VmeetsW} 
 \end{align}
By Lemmas~\ref{L:Arcneqd} and~\ref{L:Araneqbcneqd}, this will ensure that $\seq{a,U}\nearrow\orth{\seq{W}}$ and $\seq{W}\nearrow\orth{\seq{V}}$, hence that $\seq{a,U}\Dr\seq{V}$.

If $U\cap V\neq\es$, set $W=U\cap V$.
Then~\eqref{Eq:UmeetsW} and~\eqref{Eq:VmeetsW} are trivial, while
$V\cup W=V\neq[n]$ and $U\cup W=U\neq[n]$, so~\eqref{Eq:aUcmeetsWc}
and~\eqref{Eq:VcmeetsWc} are satisfied as well. If $U\cap V=\es$, then, as $(U,V)$ is not a partition of~$[n]$, we get $U^\cpl\cap V^\cpl\neq\es$; we set in this case $W=U\cup V$.
Again, it is easy to verify that
\eqref{Eq:aUcmeetsWc}--\eqref{Eq:VmeetsW} are satisfied.
\end{proof}

The description of the~$\Dr$ relation on~$\Bip(n)$ is completed by the following result.

We say that a partition $(U,V)$ of~$[n]$ is \emph{unbalanced} if either~$U$ or~$V$ is a singleton.

\begin{lemma}\label{L:DRelonBipart}
Let $U,V\in\Pow^*[n]$. Then $\seq{U}\Dr\seq{V}$ if{f} $(U,V)$ is not an unbalanced partition of~$[n]$.
\end{lemma}

\begin{proof}
  The relation $\seq{U}\Dr\seq{V}$ holds if{f} there exists a \jirr\ 
  $\bp\in\Bip(n)$ such that $\seq{U}\nearrow\orth{\bp}$ and
  $\bp\nearrow\orth{\seq{V}}$. By Lemma \ref{L:Arc=d}, $\bp$
  cannot be a clepsydra. Hence, by Lemma~\ref{L:Araneqbcneqd},
  $\seq{U}\Dr\seq{V}$ if{f} there exists $W\in\Pow[n]$ such that
 \begin{equation}\label{Eq:IntersUVWne}
 U\cap W\neq\es\,,\quad V\cap W\neq\es\,,
 \quad U^\cpl\cap W^\cpl\neq\es\,,\quad
 V^\cpl\cap W^\cpl\neq\es\,.
 \end{equation}
 Suppose first that $(U,V)$ is an unbalanced partition of~$[n]$, so
 that either $U=\set{a}$ or $V=\set{a}$, for some $a\in[n]$. In the
 first case, $U\cap W\neq\es$ implies that $a\in W$, hence $V\cup
 W=[n]$, in contradiction with~\eqref{Eq:IntersUVWne}. Likewise, in
 the second case, $a\in W$ and $U\cup W=[n]$, in contradiction
 with~\eqref{Eq:IntersUVWne}. In any case, the relation
 $\seq{U}\Dr\seq{V}$ does not hold.
 Conversely, suppose from now on that if $(U,V)$ is a partition
 of~$[n]$, then it is not unbalanced. Suppose first that $(U,V)$ is a
 partition of~$[n]$ and pick $(u,v)\in U\times V$. Then the set
 $W=\set{u,v}$ meets both~$U$ and~$V$. Furthermore, $U\cup
 W=U\cup\set{v}$ is distinct from~$[n]$ as $V$ is not a singleton.
 Likewise $V\cup W=V\cup\set{u}\neq[n]$.
 Finally, suppose that $(U,V)$ is not a partition of~$[n]$. If $U\cap
 V\neq\es$, then $W=U\cap V$ solves our problem. If $U^\cpl\cap
 V^\cpl\neq\es$, then $W=U\cup V$ solves our problem. In any case,
 \eqref{Eq:IntersUVWne} holds for our choice of~$W$.
\end{proof}

\begin{lemma}\label{L:TCDBipBip}
Suppose that $n\geq3$ and let $U,V\in\Pow^*[n]$. Then either $\seq{U}\Dr\seq{V}$ or $\seq{U}\Dr^2\seq{V}$.
\end{lemma}

\begin{proof}
By Lemma~\ref{L:DRelonBipart}, if the relation $\seq{U}\Dr\seq{V}$ fails, then there exists $a\in[n]$ such that $\set{U,V}=\set{\set{a},[n]\setminus\set{a}}$. Pick any $b\in[n]\setminus\set{a}$ and set $W=\set{a,b}$. As $n\geq3$, neither $(U,W)$ nor $(W,V)$ is a partition of~$[n]$. By Lemma~\ref{L:DRelonBipart}, it follows that $\seq{U}\Dr\seq{W}$ and $\seq{W}\Dr\seq{V}$.
\end{proof}

By using the end of Section~\ref{S:NotaTerm} (in particular Lemma~\ref{L:D2ConL}), the congruence lattice of~$\Bip(n)$ can be entirely described by the relation~$\Dr^*$ on $\Ji(\Bip(n))$. Hence it can be obtained from the following easy consequence of Corollary~\ref{C:Arc=d} together with Lemmas~\ref{L:DRelFromClep} and~\ref{L:TCDBipBip}.

\begin{corollary}\label{C:ConContBip}
Suppose that $n\geq3$ and let $\bp,\bq$ be \jirr\ elements of~$\Bip(n)$. Then $\con(\bp)\subseteq\con(\bq)$ if{f} either~$\bq$ is bipartite or~$\bp$ is a clepsydra and $\bp=\bq$.
\end{corollary}

In particular, the congruences~$\con(\bp)$, for~$\bp$ a clepsydra, are pairwise
incomparable, so they are the atoms of $\Con\Bip(n)$.  Each such congruence is thus determined by the corresponding clepsydra~$\bp$, and those clepsydras are in one-to-one correspondence with the
associated ordered pairs $(a,U\setminus\set{a})$. Hence there are
$n\cdot2^{n-1}$ clepsydras, and we get the following result.

\begin{corollary}\label{C:ConBipn}
The congruence lattice of the bipartition lattice~$\Bip(n)$ is obtained from a Boolean lattice with $n\cdot2^{n-1}$ atoms by adding a new top element, for every integer $n\geq3$.
\end{corollary}

Corollary~\ref{C:ConBipn} does not extend to the case where $n=2$: the congruence lattice of~$\Bip(2)$ is isomorphic to the lattice of all lower subsets of the poset represented on the right hand side of Figure~\ref{Fig:Bip2}.

\section{Minimal subdirect product decompositions of bipartition lattices}\label{S:SubdBip}

\begin{notation}\label{Not:Cs(n)}
For any positive integer~$n$, we define
\begin{itemize}
\item $\sG(n)$, the set of all bipartite \jirr\ elements of~$\Bip(n)$ (i.e., those of the form $U^\cpl\times U$, where $U\in\Pow^*[n]$);

\item $\sK(n)$, the \jz-subsemilattice of~$\Bip(n)$ generated by~$\sG(n)$;

\item $\theta_n$, the congruence of~$\Bip(n)$ 
generated by all~$\Psi(\bp)$, for $\bp\in\sG(n)$; 

\item $\sS(n,\bp)$, the \jz-subsemilattice of~$\Bip(n)$ generated by $\sG(n)\cup\set{\bp}$, for each $\bp\in\Ji(\Bip(n))$.

\end{itemize}
\end{notation}

By the results of Section~\ref{S:JoinBip}, the $\Dr^*$-minimal
\jirr\ of~$\Bip(n)$ are exactly the clepsydras $\seq{a,U}$, where
$a\in[n]$ and $U\subseteq[n]$ (note that the clopen set $\seq{a,U}$ is
uniquely determined by the ordered pair $(a,U\setminus\set{a})$). This
is proved in Section~\ref{S:JoinBip} for $n\geq3$, but it is also
trivially valid for $n\in\set{1,2}$ (cf. Figure~\ref{Fig:Bip2}). Hence
the minimal subdirect product decomposition of~$\Bip(n)$, given
by~\eqref{Eq:MinSubProd}, is the subdirect product
\begin{equation}\label{Eq:MinSDBipn}
  \Bip(n)\hookrightarrow\prod
  \vecm{\Bip(n)/{\Psi(\seq{a,U})}}
  {a\in[n]\,,\ U\subseteq[n]\setminus\set{a}}\,.
\end{equation}
By Lemma~\ref{L:D2ConL}, the factors of the
decomposition~\eqref{Eq:MinSDBipn} are exactly the lattices
$\Bip(n)/{\Psi(\seq{a,U})}\cong\sS(n,\seq{a,U})$. Likewise, we
can also observe that $\Bip(n)/{\theta_n}\cong\sK(n)$.  We shall now
identify, within~$\Bip(n)$, the elements of $\sS(n,\seq{a,U})$.

\begin{definition}\label{D:IsolPt}
An element $a\in[n]$ is an \emph{isolated point} of a bipartition~$\bx$ if $(a,i)\in\bx$ and $(i,a)\in\bx$ if{f} $i=a$, for each $i\in[n]$. We denote by $\isol(\bx)$ the set of all isolated points of~$\bx$.
\end{definition}

\begin{lemma}\label{L:IsolPt}
Let $\bx=\bigvee\vecm{\bx_i}{i\in I}$ in $\Bip(n)$. Then any isolated point~$a$ of~$\bx$ is an isolated point of some~$\bx_i$.
\end{lemma}

\begin{proof}
It suffices to prove that $(a,a)\in\bx_i$ for some~$i$. As~$\bx$ is
the transitive closure of the union of the~$\bx_i$, there are a
positive integer~$\ell$ and $a=a_0,a_1,\dots,a_{\ell-1}\in[n]$ such
that, setting $a_\ell=a$, the pair $(a_k,a_{k+1})$ belongs to
$\bigcup_{i\in I}\bx_i$ for each~$k<\ell$. As~$\bx$ is transitive, $(a_0,a_1),(a_1,a_0)\in\bx$, and $a=a_0$ is an isolated point of~$\bx$, we get $a_0=a_1$, so that $(a,a)$ belongs to $\bigcup_{i\in I}\bx_i$.
\end{proof}

\begin{proposition}\label{P:DescrSn}
The elements of~$\sK(n)$ are exactly the bipartitions without isolated points.
\end{proposition}

\begin{proof}
No bipartite \jirr\ element of~$\Bip(n)$ has any isolated point, hence, by Lemma~\ref{L:IsolPt}, no element of~$\sK(n)$ has any isolated point. 

Conversely, let $\bx\in\Bip(n)$ with no isolated point and let $(i,j)\in\bx$. If $i\neq j$, then, by Lemma~\ref{L:OpenUjirr}, there exists $U\in\Pow^*[n]$ such that $(i,j)\in\seq{U}\subseteq\bx$.

If $i=j$, then, as~$\bx$ has no isolated point, there exists $k\neq i$ such that $(i,k)\in\bx$ and $(k,i)\in\bx$. By the paragraph above, there are $U,V\in\Pow^*[n]$ such that $(i,k)\in\seq{U}\subseteq\bx$ and $(k,i)\in\seq{V}\subseteq\bx$. Hence $(i,i)\in\seq{U}\vee\seq{V}\subseteq\bx$. Therefore, $\bx$ is a join of clopen sets of the form~$\seq{U}$.
\end{proof}

\begin{proposition}\label{P:DescrSnaU}
Let $a\in[n]$ and let $U\subseteq[n]$. The elements of $\sS(n,\seq{a,U})$ are exactly the bipartitions~$\bx$ such that
\begin{enumerate}
\item $\isol(\bx)\subseteq\set{a}$,

\item if $\isol(\bx)=\set{a}$, then $U^\cpl\times\set{a}$ and $\set{a}\times U$ are both contained in~$\bx$.
\end{enumerate}
\end{proposition}

\begin{proof}
 Every \jirr\ element of~$\Bip(n)$ which is either bipartite or equal to $\seq{a,U}$ satisfies both~(i) and~(ii) above, hence, by Lemma~\ref{L:IsolPt}, so do all elements of $\sS(n,\seq{a,U})$.

Conversely, let $\bx\in\Bip(n)$ satisfy both~(i) and~(ii) above. If $\isol(\bx)=\es$, then, by Proposition~\ref{P:DescrSn}, $\bx\in\sK(n)$, hence, \emph{a fortiori}, $\bx\in\sS(n,\seq{a,U})$.

Now suppose that $\isol(\bx)=\set{a}$.
It follows from~(ii) together with the transitivity of~$\bx$ that $\seq{a,U}\subseteq\bx$. For each $(i,j)\in\bx\setminus\set{(a,a)}$, it follows from the argument of the proof of Proposition~\ref{P:DescrSn} that there are $U,V\in\Pow^*[n]$ such that $(i,j)\in\seq{U}\vee\seq{V}\subseteq\bx$. Hence $\bx\in\sS(n,\seq{a,U})$.
\end{proof}

Proposition~\ref{P:DescrSnaU} makes it possible to identify the
factors of the minimal subdirect product
decomposition~\eqref{Eq:MinSDBipn} of the bipartition
lattice~$\Bip(n)$. Observe that a similar result is established, in
Santocanale and Wehrung~\cite{SaWe11}, for the
\emph{permutohedron}~$\sP(n)$. In that paper, it is proved, in
particular, that the corresponding subdirect factors are exactly the
\emph{Cambrian lattices of type A} (cf. Reading~\cite{Read06}),
denoted there by~$\sA_U(n)$, for $U\subseteq[n]$.

Hence the lattices $\sS(n,\seq{a,U})$ can be viewed as the analogues,
for bipartition lattices, of the Cambrian lattices of type~A (i.e.,
the~$\sA_U(n)$). For either $U=\nobreak\es$ or $U=[n]$ (the
corresponding lattices are isomorphic, \emph{via}
$\bx\mapsto\bx^\op$), we get the bipartition analogue of the
\emph{Tamari lattice} $\sA(n)=\sA_{[n]}(n)$ (cf. Santocanale and Wehrung~\cite{SaWe11}), namely
$\sS(n,\seq{a,\es})$ (whose isomorphism class does not depend on~$a$).

\begin{proposition}\label{P:SnkSelfDual}
Every lattice $\sS(n,\seq{a,U})$ is both isomorphic and dually isomorphic to $\sS(n,\seq{a,U^{\cpl}})$. In particular, each $\sS(n,\seq{a,U})$ is self-dual.
\end{proposition}

\begin{proof}
Writing $\bp=\seq{a,U}$, the clopen set $\tilde{\bp}=\seq{a,U^\cpl}$
depends only on~$\bp$. An isomorphism from
$\sS(n,\bp)$ onto $\sS(n,\tilde{\bp})$ is given by the mapping sending a
relation~$\bx$ to its opposite~$\bx^\op$. 
We argue next that $\sS(n,\bp)$ is dually isomorphic to $\sS(n,\tilde{\bp})$. To this goal, denote by~$\sM(\bp)$ the $(\wedge,1)$-subsemilattice of~$\Bip(n)$ generated by the set
 \[
 \setm{\bu\in\Mi\Bip(n)}
 {(\bu,\bu^*)\notin\Psi(\bp)}\,.
 \]
It follows from Lemma~\ref{L:D2ConL}, applied to the dual lattice of~$\Bip(n)$, that $\Bip(n)^\op/{\Psi(\bp)}$ is isomorphic to~$\sM(\bp)$ endowed with the dual ordering of~$L$; hence,
 \[
 \sM(\bp)\cong\sS(n,\bp)
 \cong\Bip(n)/{\Psi(\bp)}\,.
 \]
On the other hand, $\sM(\bp)$ is dually isomorphic, \emph{via} the operation $\bx\mapsto\orth{\bx}$ of complementation on~$\Bip(n)$, to the \jz-subsemilattice~$\sS'(\bp)$ of~$\Bip(n)$ generated by the subset
 \[
 \sG'(\bp)=\setm{\br\in\Ji(\Bip(n))}
 {(\orth{\br},\orth{(\br_*)})\notin\Psi(\bp)}\,.
 \]
Now for each $\br\in\Ji(\Bip(n))$,
 \begin{align*}
 (\orth{\br},\orth{(\br_*)})\notin\Psi(\bp)&\Leftrightarrow
 (\sG(n)\cup\set{\bp})\dnw\orth{\br}\neq
 (\sG(n)\cup\set{\bp})\dnw\orth{(\br_*)}
 &&(\text{by Lemma~\ref{L:D2ConL}})\\
 &\Leftrightarrow(\exists\bq\in\sG(n)\cup\set{\bp})
 (\bq\leq\orth{(\br_*)}\text{ and }\bq\nleq\orth{\br})\\
 &\Leftrightarrow(\exists\bq\in\sG(n)\cup\set{\bp})
 (\bq\nearrow\orth{\br})\,. 
 \end{align*}
If $\br$ is bipartite, then, by Lemma~\ref{L:Araneqbcneqd}, there is always $\bq\in\sG(n)$ such that $\bq\nearrow\orth{\br}$ (for example $\bq=\br$). Now suppose that $\br=\seq{b,W}$ is a clepsydra. By Lemma~\ref{L:Arc=d}, $\bq\nearrow\orth{\br}$ can occur only in case $\bq=\seq{b,W^\cpl}$; furthermore, this element belongs to $\sG(n)\cup\set{\bp}$ if{f} $a=b$ and $U\setminus\set{a}=W^\cpl\setminus\set{a}$ (cf. Lemma~\ref{L:TwoJirrEq}). Therefore,
 \[
 \sG'(\bp)=\sG(n)\cup\set{\tilde{\bp}}\,,
 \]
and therefore $\sS'(\bp)=\sS(n,\tilde{\bp})$ is dually isomorphic to~$\sS(n,\bp)$.
\end{proof}

While the lattices~$\sA(n)$ and $\sA_U(n)$ are lattice-theoretical
retracts of~$\sP(n)$ (this originates in Bj\"orner and Wachs
\cite{BjWa97} and is stated formally in Santocanale and
Wehrung~\cite{SaWe11}), the situation for the
``bipartition-Tamari lattice'' $\sS(n,\seq{a,\es})$ is, as the
following examples show, not so nice.

\begin{example}\label{Ex:meetnotinters}
It is proved in Santocanale and Wehrung~\cite{SaWe11} that the meet in any Cambrian lattice of type~$A$ (i.e., any $\sA_U(n)$) is the same as set-theoretical intersection. This result does not extend to $\sS(n,\seq{a,\es})$. Indeed, both bipartitions
 \begin{align*}
 \ba&=\set{(1,1),(1,3),(2,1),(2,3),(3,1),(3,3)}\,,\\
 \bb&=\set{(1,1),(1,2),(1,3),(3,1),(3,2),(3,3)}
 \end{align*}
belong to $\sK(3)$ (use Proposition~\ref{P:DescrSn}), thus to $\sS(3,\seq{2,\es})$, but their intersection,
 \[
 \ba\cap\bb=\set{(1,1),(1,3),(3,1),(3,3)}
 \]
has empty interior, so the clopen set $\ba\wedge\bb=\es$ (in any of the lattices~$\sK(3)$, $\sS(3,\seq{2,\es})$, $\Bip(3)$) is distinct from $\ba\cap\bb$.
\end{example}

\begin{example}\label{Ex:NotRetr}
The result of~\cite{SaWe11}, stating that every Cambrian lattice~$\sA_U(n)$ is a lattice-theoretical retract of the corresponding permutohedron~$\sP(n)$, does not extend to bipartition lattices. For example, both bipartitions
 \begin{align*}
 \ba&=\set{(1,1),(1,3),(2,1),(2,3),(3,1),(3,3)}\,,\\
 \bc&=\set{(1,1),(1,2),(2,1),(2,2),(3,1),(3,2)}
 \end{align*}
belong to $\sK(3)$ (use Proposition~\ref{P:DescrSn}), thus to $\sS(3,\seq{2,\es})$, and their meet in~$\Bip(3)$ is also their intersection,
 \[
 \ba\wedge\bc=\ba\cap\bc=
 \set{(1,1),(2,1),(3,1)}\,,
 \]
 which does not belong to $\sS(3,\seq{2,\es})$. Hence, $\sS(3,\seq{2,\es})$ is not a sublattice of~$\Bip(3)$.
\end{example}

It is noteworthy to observe that, for a fixed positive integer~$n$,
the cardinality of the Cambrian lattice $\sA_{U}(n)$ is equal to the
Catalan number $\frac{1}{n+1}\binom{2n}{n}$, so it is independent
of~$U$ (combine Reading \cite[Theorem~9.1]{Read07a} with Reading
\cite[Theorem~1.1]{Read07b}: the former gives the enumeration of all
sortable elements, while the latter says that the sortable elements
are exactly the minimal elements of the Cambrian congruence
classes. This argument is established there for all finite Coxeter
groups). The following example 
%%(see also Problem \ref{Pb:CardSnk})
emphasizes that the situation is quite different for the subdirectly
irreducible factors $\sS(n,\seq{a,U})$ of $\Bip(n)$, whose size might
depend on~$\card U$.
  
\begin{example}\label{Ex:OrdersCleps}
  Set $\sS(n,k)=\sS(n,\seq{1,\set{2,3,\dots,k+1}})$, for all
  integers~$n$ and~$k$ with $0\leq k<n$. Observe that
  $\sS(n,\seq{a,U})\cong\sS(n,\card U)$, for each $a\in[n]$ and each
  $U\subseteq[n]\setminus\set{a}$. Furthermore, it follows from
  Proposition~\ref{P:SnkSelfDual} that
  $\sS(n,\seq{a,U})\cong\sS(n,\seq{a,[n]\setminus(U\cup\set{a})})$,
  hence $\sS(n,k)\cong\sS(n,n-1-k)$, and hence the factors of the
  minimal subdirect product decomposition~\eqref{Eq:MinSDBipn}
  of~$\Bip(n)$ are exactly the lattices~$\sS(n,k)$ where $n>0$ and $0\leq 2k<n$.

For small values of~$n$, the orders of those lattices are the following:
 \begin{gather*}
 \card\sS(3,0)=24\,;
 \quad\card\sS(3,1)=21\,.\\
 \card\sS(4,0)=158\,;
 \quad\card\sS(4,1)=142\,.\\
 \card\sS(5,0)=1{,}320\,;\quad
 \card\sS(5,1)=1{,}202\,;\quad
 \card\sS(5,2)=1{,}198\,.\\
 \card\sS(6,0)=13{,}348\,;\quad
 \card\sS(6,1)=12{,}304\,;\quad
 \card\sS(6,2)=12{,}246\,.
 \end{gather*}
Setting~$M(n)=\card\Bip(n)$, it is established in Wagner~\cite{Wagn82} that the~$M(n)$ are characterized by the induction formula
 \[
 M(0)=1\,;\quad
 M(n)=2\cdot\sum_{k=1}^n
 \binom{n}{k}M(n-k)\quad
 \text{if }n>0\,.
 \]
The first entries of the sequence of numbers~$M(n)$ are
 \[
 M(3)=74\,;\quad
 M(4)=730\,;\quad
 M(5)=9{,}002\,;\quad
 M(6)=133{,}210\,.
 \]
The sequence of numbers~$M(n)$ is \texttt{A004123} of Sloane's Encyclopedia of Integer Sequences~\cite{OEIS}. We do not know of any induction formula for the numbers $\card\sS(n,k)$, where $n>0$ and $0\leq 2k<n$.
 
The lattices~$\sS(3,0)$ and~$\sS(3,1)$ are represented on the left hand side and the right hand side of Figure~\ref{Fig:S301}, respectively.
\end{example}

\begin{figure}[htb]
  \centering
  \includegraphics[scale=0.5]{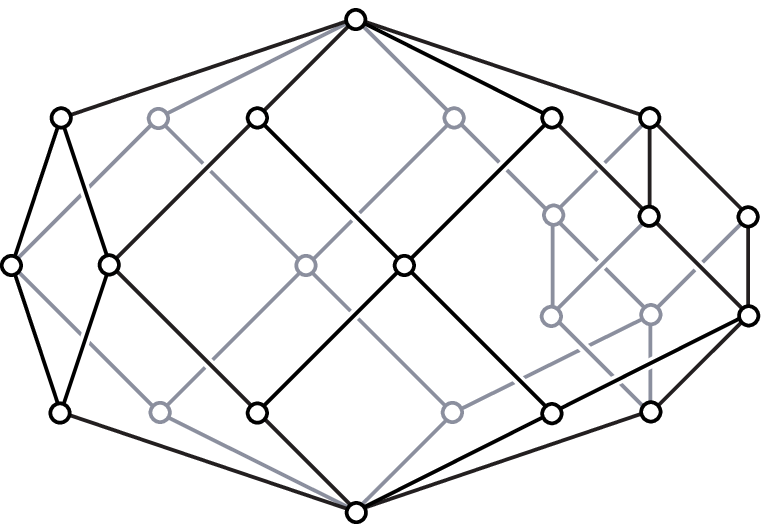}
  \hspace{10mm}
  \includegraphics[scale=0.5]{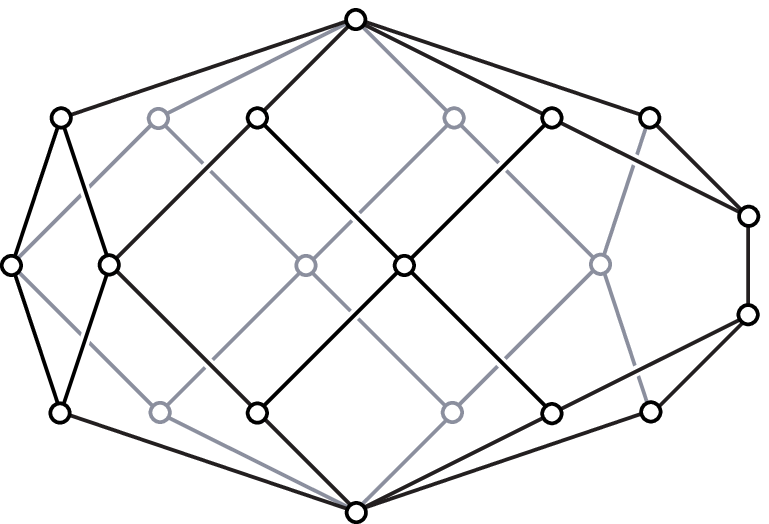}
  \caption{The bipartition-Tamari lattices $\sS(3,0)$ and $\sS(3,1)$}\label{Fig:S301}
\end{figure}

\section{Open problems}\label{S:Pbs}

Our first problem asks for a converse to Theorem~\ref{T:RegeSD}.

\begin{problem}\label{Pb:EmbPbBdedOrth}
Can every finite ortholattice, which is also a bounded homomorphic image of a free lattice, be embedded into~$\Reg(\be)$, for some finite strict ordering~$\be$?
\end{problem}

A variant of Problem~\ref{Pb:EmbPbBdedOrth}, for arbitrary finite ortholattices, is the following.

\begin{problem}\label{Pb:EmbBip}
Can every finite ortholattice be embedded into $\Bip(n)$, for some positive integer~$n$?
\end{problem}

On the opposite side of Problems~\ref{Pb:EmbPbBdedOrth} and~\ref{Pb:EmbBip}, it is natural to state the following problems.

\begin{problem}\label{Pb:EqThRegeBded}
Is there a nontrivial lattice \pup{resp., ortholattice} identity that holds in $\Reg(\be)$ for every finite strict ordering~$\be$?
\end{problem}

\begin{problem}\label{Pb:EqThRege}
Is there a nontrivial lattice \pup{resp., ortholattice} identity that holds in $\Bip(n)$ for every positive integer~$n$?
\end{problem}

Bruns observes in \cite[\S (4.2)]{Bruns76} that the variety of all ortholattices is generated by its finite members (actually, the argument presented there shows that ``variety'' can even be replaced by ``quasivariety''). This shows, for example, that Problems~\ref{Pb:EmbBip} and~\ref{Pb:EqThRege} cannot simultaneously have a positive answer.

\section{Acknowledgment}\label{S:Acknow}
The authors are grateful to William McCune for his
\texttt{Prover9-Mace4} program~\cite{McCune}, to Ralph Freese for his lattice drawing program~\cite{LatDraw}, and to the authors of the \texttt{Graphviz} framework, available online at \url{http://www.graphviz.org/}. Part of this work was completed while the second author was visiting the CIRM in March 2012. Excellent conditions provided by the host institution are greatly appreciated.

\end{document}